\documentclass
{amsart}
\usepackage{amsmath,amsthm,amsfonts,amssymb}
\usepackage[all]{xy}
\usepackage{graphicx}
\usepackage{caption}
\usepackage{subcaption}

\usepackage{amscd}
\usepackage{graphicx}
\usepackage{pb-diagram}
\usepackage{color}

\usepackage[hidelinks]{hyperref}

\hypersetup{
colorlinks=true,       
linkcolor=blue,          
citecolor=red,        
filecolor=magenta,      
urlcolor=cyan           
}

\numberwithin{equation}{section}

\newtheorem{theorem}{Theorem}[section]

\newtheorem{corollary}[theorem]{Corollary}

\newtheorem{lemma}[theorem]{Lemma}
\newtheorem{remark}[theorem]{Remark}
\newtheorem{prop}[theorem]{Proposition}

\newcommand\R{{\mathbb{R}}}

\newcommand\Z{{\mathbf{Z}}}

\newcommand\supp{\operatorname{supp}}

\newcommand\BI{{\mathbf I}}

\newcommand\CB{{\mathcal B}}
\newcommand\CC{{\mathcal C}}
\newcommand\CE{{\mathcal E}}
\newcommand\CF{{\mathcal F}}

\newcommand\CI{{\mathcal I}}

\newcommand\CK{{\mathcal K}}

\newcommand\CP{{\mathcal P}}

\newcommand\CR{{\mathcal R}}
\newcommand\CS{{\mathcal S}}

\newcommand\BBR {{\mathbb R}}
\newcommand\BBZ {{\mathbb Z}}

\newcommand\BBS{{\mathbb S}}

\newcommand{\Be}{\begin{equation}}
\newcommand{\Ee}{\end{equation}}
\newcommand{\Bea}{\begin{eqnarray}}
\newcommand{\Eea}{\end{eqnarray}}
\newcommand{\Beas}{\begin{eqnarray*}}
\newcommand{\Eeas}{\end{eqnarray*}}
\newcommand{\Benu}{\begin{enumerate}}
\newcommand{\Eenu}{\end{enumerate}}
\newcommand{\Bi}{\begin{itemize}}
\newcommand{\Ei}{\end{itemize}}

\renewcommand{\Z}{\mathbb Z}

\begin{document}
\author[E. Jeong]{Eunhee Jeong}
\author[S. Lee]{Sanghyuk Lee}
\author[A. Vargas]{Ana Vargas}

\address{Department of Mathematical Sciences, Seoul National University, Seoul 151-747, Republic of  Korea}
\email{moonshine10@snu.ac.kr}
\email{shklee@snu.ac.kr}

\address{Department of Mathematics,  Universidad Aut\'onoma de Madrid,  28049 Madrid, Spain} \email{ana.vargas@uam.es}

\title[Improved bound for the bilinear Bochner-Riesz operator]{Improved bound for the bilinear \\ Bochner-Riesz operator}

\thanks{E. Jeong   supported by NRF-2015R1A4A104167 (Republic of Korea), S. Lee  supported by  NRF-2015R1A2A2A05000956 (Republic of Korea),  and   A. Vargas supported by grants MTM2013-40945-P and MTM2016-76566-P (Ministerio de Econom\'ia y Competitividad, Spain).}

\keywords{Bilinear multiplier operator, Bochner-Riesz summability}

\maketitle 


\begin{abstract}
We study $L^p\times L^q\to L^r$ bounds for the bilinear Bochner-Riesz operator $\mathcal{B}^\alpha$, $\alpha>0$ in $\mathbb{R}^d,$ $d\ge2$, which is defined by 
\[  {\mathcal B}^{\alpha}(f,g)=\iint_{\mathbb{R}^d\times\R^d} e^{2\pi i x\cdot(\xi+\eta)} (1-|\xi|^2-|\eta|^2 )^{\alpha}_+ ~ \widehat{f}(\xi)\,\widehat{g}(\eta)\,d\xi d\eta.\] 
We make use of a decomposition which  relates the estimates for $\mathcal{B}^\alpha$ to those of the square function estimates for the classical Bochner-Riesz operators. In consequence,  we significantly improve the previously known bounds.   
\end{abstract}


\begin{section}{Introduction}
Let $d\ge2$. The Bochner-Riesz operator in $\BBR^d$ of order $\alpha\ge0$ is the multiplier operator defined by
\[\CR_t^\alpha(f)(x) =\int_{\BBR^d}e^{2\pi i x\cdot \xi}(1-|\xi|^2/t^2)^\alpha_+\widehat{f}(\xi)d\xi,\   f\in \CS(\R^d),\  t>0,\]
where $x\cdot y$ is the usual inner product in $\BBR^d$, $r_+ =r$ if $r>0$ and $r_+=0$ if $r\le 0$. Here $\CS(\R^d)$ denotes the Schwartz space in $\R^d$ and $\widehat{f} $ is the Fourier transform of $f$. Related to summability of Fourier series and integral in $L^p$,  boundedness of the Bochner-Riesz operators in $L^p$ spaces has been of interest and it is known as one of most fundamental problems in harmonic analysis which is also connected to the outstanding open problems  such as  restriction  problem  for the sphere and  the Kakeya conjecture (\cite{tao}).
 For  $1\le p\le \infty$ and $p\neq 2$,  it is conjectured that $\CR_1^\alpha$ is bounded on $L^p(\R^d)$ if and only if 
\begin{equation}
\label{eq:alpha}
\alpha >\max\Big\{d\Big|\frac12-\frac1p\Big|-\frac12, 0\Big\}.  
\end{equation}
When $\alpha=0,$ $\CR_1^\alpha$ is the disc multiplier (and ball multiplier) operator and  Fefferman \cite{f} verified that it is unbounded on $L^p(\R^d)$ except $p=2$.
For $d=2$ the conjecture was shown to be true  by Carleson and Sj\"olin \cite{c-s}, but  in higher dimensions $d\ge3$  the conjecture is verified on a restricted range and remains open.
To be more specific,  the  sharp $L^p$-boundedness of $\CR_1^\alpha$ for $p$ satisfying $\max\{p,p'\}\ge 2(d+1)/(d-1)$ follows from  the argument due to Stein \cite{f1} and the sharp $L^2$ restriction estimate for the sphere which is also known as Stein-Tomas theorem.  Subsequently,  progresses have been made by   Bourgain \cite{bourgain}, and  Tao-Vargas 
 in \cite{t-v1} when $d=3$. One of the authors \cite{l}  showed that  the conjecture holds to $\max\{p,p'\}\ge 2+4/d$. When $d\ge5$, further progress was recently made by Bourgain and Guth \cite{b-g} and  the conjecture is now verified  for $\max\{p,p'\}\ge 2+12/(4d-3-k)$ if $d\equiv k$ (mod $3$), $k=-1,0,1$.

Let $m$ be a bounded measurable function on $\mathbb R^{2d}$. Let us define the bilinear multiplier operator   $T_m$  by
\[T_m(f,g)(x)=\iint_{\R^d\times \R^d} e^{2\pi i x\cdot(\xi+\eta)} m(\xi,\eta) \widehat{f}(\xi)\widehat{g}(\eta)d\xi d\eta, \quad f,g\in \CS(\R^d).\]
As in linear multiplier case,  it  is a natural problem  to characterize $L^p\times L^q\to L^r$ boundedness of   $T_m$. The problem may be regarded as bilinear generalization of linear  one and has applications, especially, to controlling nonlinear terms in various nonlinear partial differential equations (\cite{tao-book}).  
Boundedness properties of $T_m$ are mainly determined by the singularity of  the multiplier $m$. In fact, 
if $m$ is smooth and compactly supported, then $T_m$ is bounded from $L^p\times L^q\to L^r$ whenever 
$\frac1p+\frac1q\ge \frac1r$, $p,q,r\ge 1$.

Unless $m=m_1\otimes m_2$ for some $m_1, m_2$ on $\R^d$, $L^p\times L^q\to L^r$  boundedness of $T_m$ can not generally be deduced  from that of  linear multiplier operator, and the problem is known to be substantially more difficult than obtaining boundedness for linear operator.  Most well known are Coifman-Meyer's result  on bilinear singular integrals and  the  boundedness of bilinear Hilbert transform having multiplier with singularity along a line, of which boundedness on Lebesgue space is now relatively well understood (\cite{d-t,l-t,l-t1}). The similar  bilinear operators given by multipliers with different types of singularities also have  been of interest and studied by several authors.  We refer the reader to \cite{b-f-r-v,b-t,c-m1,d-p-t,g-k,m-t,tomita} and references therein for further relevant literature.

In this note, we investigate $L^p\times L^q\to L^r$ boundedness of the bilinear Fourier multiplier operator which is called \emph{bilinear Bochner-Riesz operator. } The operator is a bilinear extension of the Bochner-Riesz operator. As in the classical Bochner-Riesz case, the boundedness of bilinear Bochner-Riesz operator has implication to convergence of  Fourier series, especially, the summability of the product of two $d$-dimensional Fourier series.  See \cite{b-g-s-y} for details.  
Let  $d\ge 1$. The bilinear Bochner-Riesz operator $\CB^\alpha$ of order $\alpha\ge0$ in $\BBR^d$ is defined  by
\begin{equation}\label{eq1}
{\mathcal B}^{\alpha}(f,g)(x) = \iint_{\mathbb{R}^d\times\R^d} e^{2\pi i x\cdot(\xi+\eta)} (1-|\xi|^2-|\eta|^2 )^{\alpha}_+ ~ \widehat{f}(\xi)\widehat{g}(\eta)d\xi d\eta
\end{equation}
for $ f,g\in \CS(\R^d).$   For simplicity we set $m^\alpha(\xi,\eta)=(1-|\xi|^2-|\eta|^2)^\alpha_+$ in what follows. 
We are concerned with the  estimate, for $ f,g\in \CS(\R^d)$, 
\begin{equation}\label{eq:b}
\|\CB^\alpha(f,g)\|_{L^r(\R^d)}\le C\|f\|_{L^p(\R^d)}\|g\|_{L^q(\R^d)}.
\end{equation}
Since $\CB^\alpha$ is commutative under simultaneous translation, 
 \eqref{eq:b} holds only 
if $1\le p,q\le \infty$ and $0<r\le \infty$ satisfies $1/p+1/q\ge 1/r.$
In view of  this,   the case in which  H\"older relation $1/p+1/q=1/r$ holds may be regarded as a critical case. 
This case is also important  since \eqref{eq:b} becomes  scaling invariant. Thus, by the standard density argument  one can deduce from   \eqref{eq:b}  the convergence 
\[\lim_{\lambda\to\infty}\iint_{\mathbb{R}^d\times\R^d} e^{2\pi i x\cdot(\xi+\eta)} \Big(1-\frac{|\xi|^2+|\eta|^2}{\lambda^2} \Big)^{\alpha}_+ ~ \widehat{f}(\xi)\widehat{g}(\eta)d\xi d\eta  = f(x) g(x)
\] 
in $L^r$ whenever $f\in L^p$ and $g\in L^q$, {$p, q\neq \infty$} .   Studies on  boundedness of $\mathcal B^\alpha$ under H\"older relation were carried out  recently by several authors \cite{g-l,d-g, b-ge, b-g-s-y}. 
When $d=1$, the problem was almost completely  solved when the involved $L^p, L^q, L^r$ are Banach spaces (see \cite[Th eorem 4.1]{b-g-s-y} and \cite{g-l,b-ge}), that is to say, all of $p,q,r$ are in $[1,\infty]$. For higher dimensions $d\ge 2$, Diestel and Grafakos \cite{d-g} proved that for $\alpha=0$ \eqref{eq:b} cannot hold if exactly one of  $ p,q,r'=r/(r-1)$ is less than $2$, by modifying  Fefferman's  counterexample to the (linear) disk multiplier conjecture \cite{f}.

Boundedness of $\mathcal B^\alpha$ for general $\alpha>0$  was studied by  Bernicot, Grafakos, Song, and Yan in \cite{b-g-s-y}. They obtained some positive and negative results for the boundedness for $\CB^\alpha$ for any $p$ and $q$ between $1$ and $\infty.$ However, to state their results in full detail is a bit complicated. So,  focussing on Banach cases, we summarize some of them in the following, which are the most recent result regarding  boundedness of $\mathcal B^\alpha$ as far as we are aware.

\begin{prop}\cite[Proposition 4.10, 4.11]{b-g-s-y}\label{thm:bg} Let $d\ge 2$ and $1\le p,q,r\le \infty$ with $1/p+1/q=1/r.$ Then \eqref{eq:b} holds if exponents $p,q,r$ and $\alpha$ satisfy one of the following conditions:
\vspace{-10pt}
\begin{itemize}
\item $2\le p,q < \infty$, $1\le r\le 2$ and $\alpha > (d-1)(1-\frac1r)$; 
\item $ 2\le p,q,r<\infty$ and $\alpha>\frac{d-1}2 +d(\frac12-\frac1r)$;
\item $2\le q<\infty$, $1\le p,r<2$ and $\alpha >d(1/2-1/q)-1/r$;
\item $2\le p<\infty$, $1\le q,r<2$ and $\alpha >d(1/2-1/p)-1/r$.
\end{itemize}
\end{prop}

In particular, $\CB^\alpha$ is bounded from $L^2\times L^2$ to $L^1$ if and only if $\alpha>0.$ 
In \cite{b-g-s-y}\,   $L^2\times L^2\to L^1$ boundedness was shown  for general bilinear multiplier operator $T_m$ of which the multiplier $m$ is bi-radial and compactly supported and satisfying some regularity condition.   The authors  took   advantage of bi-radial structure of $m$,  which makes it possible to  reduce a $2d$-dimensional symbol to  $2$-dimensional one. By verifying a minimal regularity condition for $m^\alpha$ they  showed $\mathcal B^\alpha$ is bounded from $L^2\times L^2\to L^1$ for all $\alpha>0$.  For the other exponents   $p,q,r$ they used the standard argument which has been used to prove $L^p$-boundedness for the classical Bochner-Riesz operator.  To be precise,  regarding $\CR^\alpha_1$ as a multiplier operator acting on $\mathbb R^{2d}$, they decomposed the multiplier dyadically away from the set $\{(\xi,\eta): m^\alpha(\xi,\eta)=0\}=\{(\xi,\eta): |\xi|^2+|\eta|^2=1\}$ and 
used estimates for the kernels of  bilinear multiplier operators which result from dyadic decomposition. From this, they showed that  \eqref{eq:b}  holds on a certain  range of $\alpha$ when $(p,q,r)=(1,\infty,1),(\infty,1,1),(2,\infty,2),(\infty,2,2)$, and  $(\infty,\infty,\infty)$. Then,  complex interpolation was used to obtain results for general exponents.

 However, as is well known in studies of multiplier operators of Bochner-Riesz type,  the kernel estimate alone is not  enough to show sharp results except for some specific  exponents.  Regarding such problem  the heart of matter lies in quantitative understanding of oscillatory cancellation.  In contrast with  the classical Bochner-Riesz operator of which boundedness is almost characterized by the frequency near the singularity on the sphere,  for the bilinear  Bochner-Riesz operator   we need to understand interaction between the two frequency variables $\xi, \eta$ as well as behavior related to the singularity of the multiplier  of $\mathcal B^\alpha$.  From  \eqref{eq1} it is natural to expect that  the worst scenario may arise from the contribution  near the intersection of the sets $|\xi|^2+|\eta|^2=1$ and  $\xi=-\eta$, where the oscillation effect disappears.  
 Our main novelty is in exploiting  this observation.  First, following the usual way we decompose $m^\alpha$ away from the singularity and then make further decomposition so that the interaction between two  $\xi$ and $\eta$ can be minimized. Then,  to handle the resulting operators we use  square function estimates for the Bochner-Riesz operator about which we give more details below.

There have been various works  which are related to so called \emph{bilinear approach} to various linear problems, such as  bilinear restriction estimates (see, for example \cite{t-v-v, wolff, t-v1, l, lee2, tao}). Since $\CB^\alpha$ has bilinear  structure, it seems natural to expect that  such bilinear methodology can be useful  to obtain improved bounds but this doesn't seem to work well for $\CB^\alpha$, especially, because of the interaction between  two frequencies near the set $\xi=-\eta$. This is the reason why we rely on the square function estimate instead of following the typical bilinear approach.

We now consider the square function $\mathfrak{S}^\alpha$ for the Bochner-Riesz means, which is defined by 
\[
 \mathfrak{S}^\alpha f(x) =\Big(\int_0^\infty \Big|\frac{\partial}{\partial t} \CR^\alpha_tf(x)\Big|^2 t dt\Big)^{1/2}.
 \]
This  was introduced by Stein \cite{stein} in order to study pointwise convergence of the Bochner-Riesz means and finding the optimal $\alpha$ for which the estimate 
\begin{equation}
 \label{square}
  \|\mathfrak S^\alpha
\!f\|_p\le C \|f\|_p\,
\end{equation}
holds 
has been investigated and it is related to various problems. See Cabery-Gasper-Trebels \cite{c-g-t} and Lee-Rogers-Seeger \cite{LRV}. The  estimate \eqref{square} is well understood for $1<p\le 2$. For $p>2$, however, it was conjectured that  $\mathfrak{S}^\alpha$ is bounded on $L^p(\R^d)$ if and only if  $\alpha >\max\{d(1/2-1/p),1/2\}.$ When $d=2$ the conjecture was proved by Carbery \cite{c}, and in higher dimensions partial results 
are known (see \cite{christ, LRV, lee1})  and   the best known results can be found in   \cite{LRV, lee1}.    

Let $0<\delta\ll 1$, $\phi$ be a smooth function supported in $[-1,1]$,  and define a square function with localized frequency which is given by 
\begin{equation}\label{lo-square}
  \mathfrak{S}_\delta^\phi f(x) = \Big(\int_{1/2} ^2\Big|\phi\Big(\frac{|D|^2-t}\delta\Big) f(x)\Big|^2dt\Big)^{1/2}.
  \end{equation}
The conjectured $L^p$ ($2<p\le \infty$)  estimate for 
$\mathfrak{S}^\alpha$ is essentially equivalent to  the following:   For $p\ge\frac{2d}{d-1}$ and $\epsilon>0$,  there exists 
$C=C(\epsilon)$ such that 
\Be
\label{del-square}
 \|\mathfrak{S}_\delta^\phi f\|_p\le C\delta^{\frac{2-d}2+\frac dp-\epsilon}\|f\|_p\, . 
 \Ee
Implication from \eqref{del-square} to \eqref{square} is easy to see from dyadic decomposition and using easy $L^2$ estimate 
$\|\mathfrak{S}_\delta^\phi f\|_2\le C\delta^{\frac12}\|f\|_2$ and interpolation.  
We don't draw direct connection from \eqref{eq:b} to \eqref{square}. Instead we show that the estimate  \eqref{eq:b}  can be deduced from  $L^p$ bound for $\mathfrak{S}_\delta^\phi$. 

To present our results, we introduce some notations: For $\nu\in [0, 1/2]$, we set  
\begin{align*}
\Delta_1(\nu)=& \big \{(u,v)\in [0,1/2]^2 : u,v\le \nu \big \},
\quad \Delta_2(\nu)= \big \{(u,v)\in  [0,1/2]^2 :  u, v\ge \nu \big \},
\\
\Delta_3(\nu)=&
\big \{(u,v)\in  [0,1/2]^2 : u < \nu <v \mbox{ or } v < \nu< u\big \}.
\end{align*} 
\begin{figure}
\centering
\includegraphics[height=6cm]{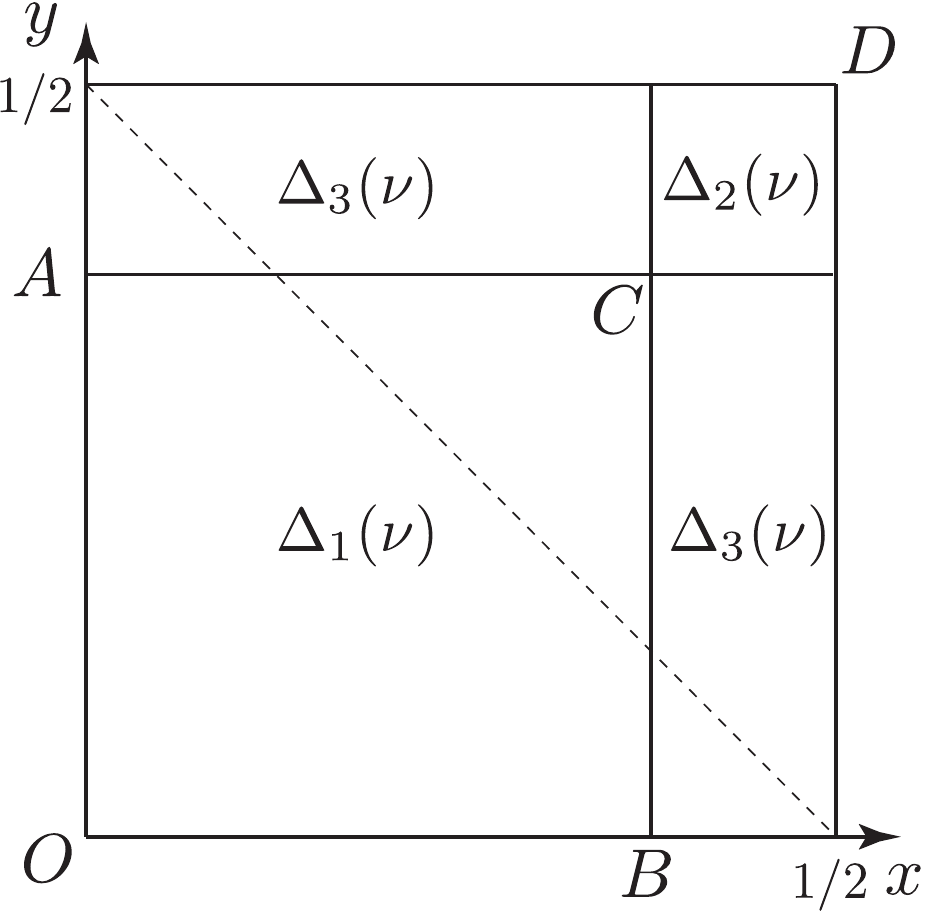}
\caption{The points $A=(0,\nu)$, $B=(\nu,0)$, $C=(\nu,\nu)$,  $D=(\frac12,\frac12)$, and $O=(0,0)$. }
\label{fig1}
\end{figure}
The regions $\Delta_j(\nu)$, $1\le j\le 3$, are pairwise disjoint and $\bigcup_{j=1}^3 \Delta_j(\nu)=[0,1/2]^2$ (see Figure \ref{fig1}). 
For $u\in [0,1]$ set \[\beta_\ast(u)=\frac{d-1}2-ud\,.\] Let us define a real valued function $\alpha_\nu: [0,1/2]^2\to \mathbb R$  by 
\[\alpha_\nu(u,v)\!
= \! \begin{cases}      
\beta_\ast(u)+\beta_\ast (v)= (d-1)-d(u+v), &   \!\!  (u,v)\in \Delta_1(\nu),  
\\
\frac{2-2u-2v}{1-2\nu}\beta_\ast(\nu),& \!\! (u,v)\in \Delta_2(\nu), 
\\
\max\{\beta_\ast(u),\beta_\ast(v)\}+ \beta_\ast(\nu)\min\{\frac{1-2u}{1-2\nu},\frac{1-2v}{1-2\nu}\},  &  \!\!  (u,v)\in \Delta_3(\nu).
\end{cases}    \]
The following is our firs result.

\begin{theorem}\label{thm:main}
 Let $d\ge 2$,  $ p_\circ\ge 2d/(d-1)$ and let $2\le p,q\le\infty$ and $r$ with $1/r=1/p+1/q.$  Suppose that for $p\ge p_\circ$ the estimate \eqref{del-square} holds with $C$ independent of $\phi$  {whenever $\phi\in \CC_N([-1,1])$ for some positive integer $N$}.  Here $\CC_N([-1,1])$ is defined by \eqref{def-cn}.  Then for any $\alpha >\alpha_{\frac1{p_\circ}}(1/p,1/q)$ \eqref{eq:b} holds.
  \end{theorem}

For $d\ge 2$ we set $p_0(d)$ and $p_s$ to be $p_0(d)=2+\frac{12}{4d-6-k},~d\equiv k ~(\mbox{mod } 3),~k=0,1,2,$
\begin{equation}\label{eq:p}
 p_s=p_s(d)=\min\Big\{p_0(d),\frac{2(d+2)}{d}\Big\}. 
  \end{equation}
We will prove (Lemma \ref{lem:lee'}) that \eqref{del-square} holds for $p\ge p_s$. Hence, this and Theorem \ref{thm:main}  yields the following.

\begin{corollary}\label{main}   Let $d\ge 2$, and let $2\le p,q\le\infty$ and $r$ be given by $1/r=1/p+1/q.$ Then \eqref{eq:b} holds provided that  $\alpha >\alpha_{\frac1{p_s}}(1/p,1/q)$ 
\end{corollary}

 Remarkably,  when $d=2$  Corollary \ref{main} gives sharp estimates for some $p,q$ other than $p=q=2$.  Indeed, note that $p_s(2)=4$ and, thus, for $2\le p, q\le 4$ we have $\alpha_{\frac14}(\frac1p,\frac1q)=0$.  By Corollary \ref{main} it follows that \eqref{eq:b} holds for $\alpha>0$  if $(p,q)\in [2,4]^2$.   This result is clearly sharp  in view of Diestel-Grafakos's result \cite{d-g}.

 Corollary \ref{main} provides improved  estimates over those in  Proposition  \ref{thm:bg}  except the case $p=2$ and $q=2$.    This can be clearly seen by considering the boundedness of  $\CB^\alpha$ from $L^p(\R^d)\times L^p(\R^d)$ to $L^{p/2}(\R^d)$. See Figure \ref{fig2}.  However,  we do not know whether  the exponents  in Corollary \ref{main} are sharp for most of the cases and we are only able to provide improved lower bounds for $\alpha$  which is slightly  better than the one known before. (See  the section \ref{necessary}.)

\begin{figure}
\centering
\begin{subfigure}{0.5\textwidth}
\centering
\includegraphics[height=6cm]{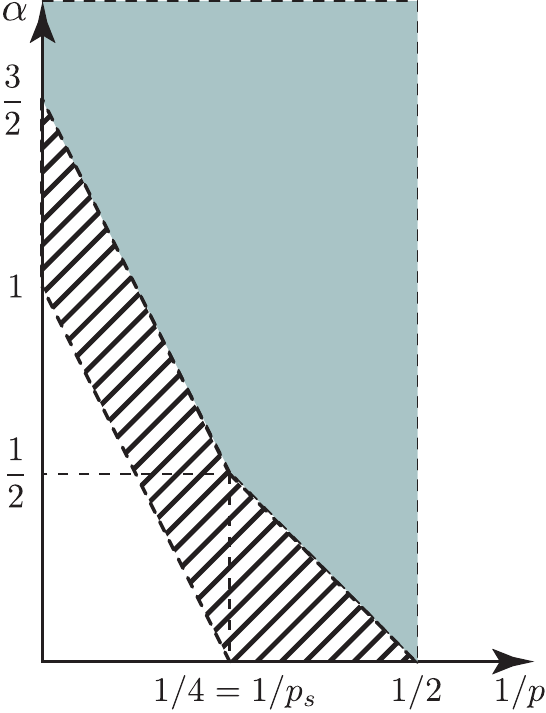}
\caption{$d=2$}
\end{subfigure}%
\begin{subfigure}{0.5\textwidth}
\centering
\includegraphics[height=6cm]{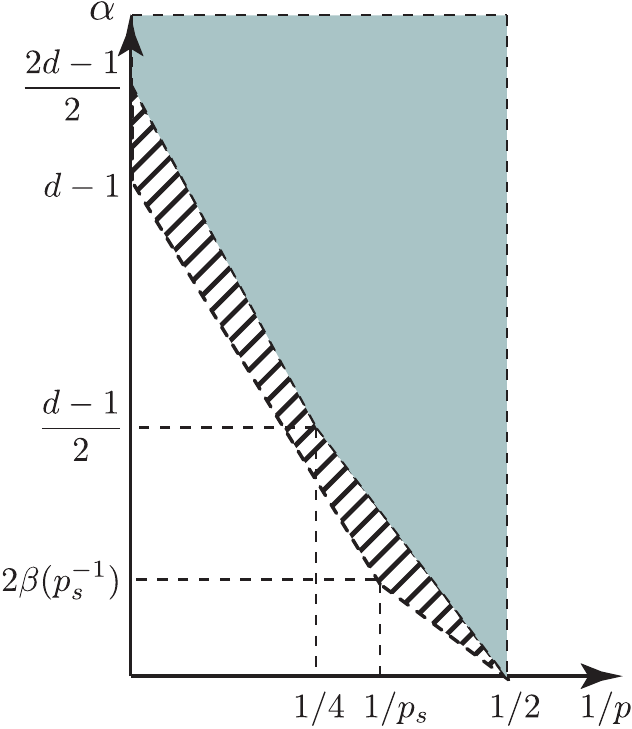}
\caption{$d\ge3$}
\end{subfigure}
\caption{The range of $\alpha$ and $p$ for $\CB^\alpha: L^p\times L^p\to L^{\frac p2}$, $d\ge2$. Proposition \ref{thm:bg} gives boundedness for $(1/p,\alpha)$ in the shaded region but Theorem \ref{thm:main} extends it the slashed region.}
\label{fig2}
\end{figure}

The main new idea  of this work is a decomposition lemma (Lemma \ref{Lemm:reduction}) which enables us to  split frequency interaction between two variables $\xi$ and $\eta$.  The decomposition lemma basically reduces the problem to dealing with the operator $\CB_{\delta,\varrho}^{\phi_1,\phi_2}$ which is  a sum of products of two linear operators with localized frequency. See \eqref{op:B} for the precise definition of $\CB_{\delta,\varrho}^{\phi_1,\phi_2}$.  This lemma  makes the problem  much simpler.  For example,  various previous result  can be easily obtained by making use of the lemma. Moreover, $S_{\rho,\delta}^\phi$ (see \eqref{eq:S}) appeared in $\CB_{\delta,\varrho}^{\phi_1,\phi_2}$ are closely  related to the (linear) Bochner-Riesz operator $\CR_t^\alpha$ and its bounds are now better understood.  Since  $\CB_{\delta,\varrho}^{\phi_1,\phi_2}$ has product structure,  by Cauchy-Schwarz inequality we can simply bounds this with a product of discretized square function  $\mathfrak D_\delta^\phi$ defined by \eqref{two}, of which  sharp bounds can be deduced from the well-known estimates for the square function  $\mathfrak S_\delta^\phi$.    

The rest of this paper is organized as follows. In Section 2 we consider two different types of square functions $\mathfrak{S}^\phi_\delta$ and $\mathfrak D_\delta^\phi$ and make observation that  their $L^p$-boundedness properties are more or less equivalent. In Section 3 we introduce a decomposition lemma which convert our problem to estimates for bilinear operators $\CB_{\delta,\varrho}^{\phi_1,\phi_2}$. In Section 4 we prove Theorem \ref{thm:main} and discuss the boundedness \eqref{eq:b} for $\CB^\alpha$ under sub-critical relation $1/p+1/q> 1/r$.  Finally, in Section 4 we find a new lower bound for $\alpha$.

Throughout the paper, the positive constant $C$ may vary line to line. For $A,B>0,$ by $A\lesssim B$, we mean $A\le CB$ for some constant $C$ independent of $A,~B$. We write $A\sim B$ to denote $A\lesssim B$ and $A\gtrsim B.$ Also, $\widehat{f}$ and $f^\vee$ denote the Fourier and inverse Fourier transforms of $f$, respectively: $
\widehat{f}(\xi)=\int_{\mathbb{R}^d}e^{-2\pi i x\cdot\xi} f(x)dx ,$ $ f^\vee(x)=\int_{\mathbb{R}^d}e^{2\pi i x\cdot\xi}f(\xi)d\xi.
$
We also use  $\CF(f)$ and $\CF^{-1}(f)$ for the Fourier and the inverse Fourier transforms of $f$, respectively. For a bilinear operator $T$ we denote by $\|T\|_{L^p\times L^q\to L^r}$ the operator norm of $T$ from $L^p(\R^d)\times L^q(\R^d)$ to $L^r(\R^d)$. 
\end{section}

\renewcommand{\CI}{I}
\begin{section}{Preliminaries}
In this section we obtain several preliminary results which we need in the course of proof. 

Let $\CI\subset \R$ be an interval, and $N$ be nonnegative integer. 
We define $\CC_N(\CI)$ to be a class of smooth functions $\phi$ on $\R$ satisfying
 \Be \label{def-cn}
  \supp\phi \subset \CI\quad\mbox{and}\quad  \sup_{t\in \CI} |\phi^{(k)}(t) | \le 1,\quad k=0,\cdots,N.
  \Ee
For a smooth $\phi$ and $0<\delta\ll 1$, we define linear operators $S_{\rho, \delta}^\phi$ by
\begin{equation}\label{eq:S}
 \widehat{S_{\rho,\delta}^\phi f}(\xi)= \phi\Big(\frac{|\xi|^2-\rho}{\delta}\Big)\widehat{f}(\xi),\quad f\in \CS(\R^d).
\end{equation}

\subsection{Kernel estimate}
 For $\omega\in \R^d$ with $|\omega|=1$ and $0<l \le 1$, let $\chi^{\omega}_l\in  C^\infty(\mathbb R^d\setminus\{0\}) $ be  a  homogeneous function  of degree $0$  such that  $\chi^{\omega}_l$ is supported in
$\Gamma_l^{\omega} := \{ \xi :  |\xi/|\xi| -\omega | \le 2l\}$ and 
 \[ |\partial_{\xi}^{\alpha} \chi^{\omega}_l(\xi)| \le C_{\alpha}l^{-|\alpha|}|\xi|^{-|\alpha|}\]
for all multi-indices $\alpha$.  We also set 
\begin{equation}
 K^{\omega,l}_{\rho,\delta}(x) = \int_{\R^d} e^{2\pi i x\cdot \xi} \phi\Big(\frac{|\xi|^2-\rho}{\delta}\Big) \chi^{\omega}_l(\xi) d\xi \label{kernel}.
\end{equation}
\begin{lemma}\label{Lemm:kernel}
Let $d\ge2$, $0<\delta\ll1,~2\delta\le \rho\le 1$. 
Suppose that $l\sim (\delta/\rho)^{1/2}$. Then there is a constant  $C$, independent of $\delta,\rho,\omega,$ such that
\begin{equation}\label{eq:k-es}
 |K^{\omega,l}_{\rho,\delta}(x)| \le C\rho^{-1/2}\delta^{(d+1)/2}(1+\delta^{1/2}|x- (\omega\cdot x) \omega|+\delta\rho^{-1/2}|\omega\cdot x| )^{-N}
 \end{equation}
  whenever $\phi\in \CC_N([-1,1])$. 
\end{lemma}
\begin{proof}  By scaling $\xi\to \sqrt \rho\xi$  and $x\to \rho^{-1/2} x$, it is sufficient to show that  \[  |K^{\omega,l}_{1,l^2}(x)| \le Cl^{d+1}(1+l^{1}|x- (\omega\cdot x) \omega|+l^2 |\omega\cdot x| )^{-N}\]
with $C$ independent of $\phi$. 
And this can be  obtained by routine  integration by parts.
\end{proof}

Making use of a homogeneous  partition of unity which is given  by $\{\chi^\omega_l\}$ with $l\sim (\delta/\rho)^{1/2}$ and $\{\omega\}$ which is a $\sim (\delta/\rho)^{1/2}$ separated subset of $\mathbb S^{d-1}$   and Lemma \ref{Lemm:kernel}, one can easily obtain the following.  

\begin{lemma}\label{lem:kernel} For $0<\delta\ll 1$,  $\rho\ge 0$, and $\phi \in \CC_{N}([-1,1])$, let us set 
\[ K_{\rho,\delta}^\phi =\mathcal F^{-1}\Big(   \phi\big(\frac{|\xi|^2-\rho}{\delta}\big)\Big) .\]  
Then,  there exists a constant $C$, independent of $\phi$ (also $\rho$ and $\delta$), such that 
\[ |K_{\rho,\delta}^\phi(x)|  \le C 
\begin{cases}
\rho^{\frac{d-2}{2}} \delta(1+\rho^{-\frac12}\delta|x|)^{-N},  &\mbox{ for } \rho\ge C\delta \\
\delta^{\frac d2}(1+\delta^{\frac12}|x|)^{-N}, &\mbox{ for } \rho\le C\delta
\end{cases}
\] 
for some large  $C>1$.  In particular,  $ |K_{\rho,\delta}^\phi(x)|  \le C \delta (1+\delta|x|)^{-N}$ for  all $\rho\in [0,1]$ .
\end{lemma}

\subsection{Discretized square function}
For  a compactly supported  smooth function $\phi$ and $0<\delta\ll 1,$ we define a discrete square functions $\mathfrak{D}_{\delta}^\phi$  by
\begin{equation}\label{two}
\mathfrak{D}_{\delta}^\phi f(x) =\Big(\sum_{\rho\in \delta\Z\cap [1/2,1]}|S_{\rho,\delta}^\phi f(x)|^2\Big)^{1/2},
\end{equation}
and let $\mathfrak{S}^\phi_\delta$ be defined by \eqref{lo-square}. In what follows we show that, for $p\ge 2$, $L^p$-boundedness properties of these two square functions are essentially equivalent.

\begin{lemma}\label{lem:sq1} Let $1\le p\le \infty$, $N$ be a positive integer, and $0<\delta\le \delta_0 \le 1/8$.  Suppose that 
\Be \label{sqsq}\Big\| \Big( \int_{\frac12}^2 |S_{t,\delta}^\phi f(x)|^2dt\Big)^{1/2}
\Big\|_{p}\le A  \|f\|_p
\Ee 
holds with $A$ independent of $\phi$ whenever $\phi\in \CC_N([-1,1])$.  Then $\|\mathfrak{D}^\phi_\delta  f \|_{p}\le 2\delta^{-\frac12}A \|f\|_p$ holds  
whenever  $\phi\in \CC_{N+1}([-1,1])$.
\end{lemma}

\newcommand{\fS}{\mathfrak S}

\begin{proof} We fix $\phi\in \CC_{N+1}([-1,1])$. From the fundamental theorem of calculus  we have
\begin{align*}
\phi\Big(\frac{|\xi|^2-\rho}\delta\Big)
&=\phi\Big(\frac{|\xi|^2-\rho-t}\delta\Big)+\delta^{-1}\int_0^t \phi'\Big(\frac{|\xi|^2-\rho-\tau}\delta\Big) d\tau.
\end{align*}
Thus  $S^\phi_{\rho,\delta}f= S^\phi_{\rho+t,\delta}f +\delta^{-1}\int_0^t S_{\rho+\tau,\delta}^{\phi'}f d\tau$.  Using this and taking additional integration in $t$ over $[0,\delta]$ give
\[S^\phi_{\rho,\delta}f(x)=\delta^{-1} \int_0^\delta S^\phi_{\rho+t,\delta}f(x) dt  +\delta^{-2}\int_0^\delta \int_0^t S_{\rho+\tau,\delta}^{\phi'}f(x) d\tau d t.\]
By  Cauchy-Schwarz  and  triangle inequalities, we have
\begin{equation}\label{eq:point}
\mathfrak{D}^\phi_\delta f(x)
\le\delta^{-1/2}\Big(\CI_1+  \delta^{-1}\CI_2\Big), 
\end{equation}
where
\begin{align*}
\CI_1=\Big(\!\!\!\sum_{\rho\in\delta\Z\cap[\frac12,1]}\! \int_0^\delta|S^\phi_{\rho+t,\delta}f(x)|^2dt  \Big)^{\frac12}, \,\,
\CI_2=\Big(\!\!\!\sum_{\rho\in\delta\Z\cap[\frac12,1]}\int_0^\delta\Big|\int_0^t S^{\phi'}_{\rho+\tau,\delta}f(x)d\tau\Big|^2dt  \Big)^{\frac12}.
\end{align*}
Then, it is clear that 
\[\CI_1=\Big(\sum_{\rho\in\delta\Z\cap[1/2,1]}\int_\rho^{\rho+\delta} |S^\phi_{t,\delta}f(x)|^2dt\Big)^{\frac12}  \le \Big(\int_{1/2}^2|S^\phi_{t,\delta}f(x)|^2dt\Big)^{\frac12}.\]
Applying H\"older's inequality to the inner integral of $\CI_2$ yields
\[
\CI_2\le  \Big(\sum_{\rho\in\delta\Z\cap[1/2,1]}\int_0^\delta t\int_0^\delta |S^{\phi'}_{\rho+\tau,\delta}f(x)|^2 d\tau dt\Big)^{1/2}\le  \delta\Big(\int_{1/2}^2|S^{\phi'}_{t,\delta}f(x)|^2dt\Big)^{1/2}.
\]
Thus, combinning this with \eqref{eq:point} 
we have $\mathfrak{D}^\phi_\delta f(x) \le \delta^{-1/2}(\fS^\phi_\delta f(x)+\fS^{\phi'}_\delta f(x)).$ Since $\phi,\phi'\in \CC_N([-1,1])$,
\[ \|\mathfrak{D}^\phi_\delta f\|_p\le \delta^{-\frac12} (\|\fS^\phi_\delta f\|_p+\|\fS^{\phi'}_\delta f\|_p)\le 2A \delta^{-\frac12}\|f\|_p.\]
This completes the proof.
\end{proof}

The implication in  Lemma \ref{lem:sq1} is reversible for a certain range of $p$.  We record the following  lemma  even though we do not use it in this paper. 
 
\begin{lemma}\label{lem:sq2} Let $2\le p\le \infty$, $N$ be a positive integer, and $0<\delta\le \delta_0 \le 1/8$. Suppose that $\|\mathfrak{D}^\phi_\delta  f \|_{p}\le A \|f\|_p$ holds with $A$ independent of $\phi$ whenever  $\phi\in \CC_{N}([-1,1])$.  Then, there is a constant $C$, {independent of $\delta$ and $\phi$}, such that $\|\mathfrak{S}^\phi_\delta f\|_p\le CA\delta^{1/2}\|f\|_p$ holds for all $\phi \in \CC_N([-1/2,1/2])$. 
\end{lemma}

Decomposing $\phi$ into functions supported in smaller intervals we may replace the interval $[-1/2,1/2]$ with $[-1,1]$.

\begin{proof}
Let $\phi \in \CC_N([-1/2,1/2])$.  To begin with, observe that 
\begin{equation}
\label{scale-scale}
S_{\rho,\delta}^\phi f(x)=   S_{\lambda^2\rho,  \lambda^2\delta }^\phi \big(f( \lambda \cdot)\big)(\lambda^{-1} x)
\end{equation}
Thus decomposing the interval $[1/2, 1]$ into finite subintervals and using the above rescaling identity it is sufficient to show that 
\[\Big\| \Big( \int_{\frac 58}^{\frac 78}|S_{t,\delta}^\phi f(x)|^2dt\Big)^{1/2}
\Big\|_{p} \le  A\delta^{\frac12} \|f\|_p. \] 

Since $\delta<1/8$ we note that  
\[ \int^{\frac78}_{\frac58}|S^\phi_{t,\delta}f(x)|^2dt\le \int_{-\delta/2}^{\delta/2}\sum_{\rho\in\delta\Z\cap[1/2,1]}|S^\phi_{\rho+t}f(x)|^2dt.\]
 For $|t|\le \delta/2,$ set $\psi_t(s)=\phi(s-\frac{t}\delta)$. Then we see $\psi_t\in \CC_N([-1,1])$ and $S^{\phi}_{\rho+t,\delta}f(x)=S^{\psi_t}_{\rho,\delta}f(x).$ Hence
\[   \int^{\frac78}_{\frac58}|S^\phi_{t,\delta}f(x)|^2dt\le  \int_{-\delta/2}^{\delta/2} \sum_{\rho\in \delta\Z\cap[1/2,1]} |S_{\rho,\delta}^{\psi_t}f(x)|^2 dt.\]
Since $p\ge 2$  and $\psi_t\in \CC_N([-1,1])$, by Minkowski's inequality and the assumption, we have
\begin{align*}
\Big\| \Big( \int^{\frac78}_{\frac58} |S_{t,\delta}^\phi f(x)|^2dt\Big)^{1/2}
\Big\|_{p} \le \Big( \int_{-\delta/2}^{\delta/2} \| \mathfrak{D}_{\rho,\delta}^{\psi_t}f\|_p^2dt\Big)^{1/2}
\le A\Big(\int_{-\delta/2}^{\delta/2}\|f\|_p^2dt\Big)^{1/2}.
\end{align*}
This gives the desired bound. 
\end{proof}

\subsection{Estimates for $\mathfrak{S}^\phi_\delta$}\label{sub:sq}   
Let $I=[-1,1]$ and set $\CE(N)$ to be a class of smooth functions $\eta\in C^\infty(I^d\times I)$ satisfying $\|\eta\|_{C^N(I^d\times I)}\le 1$ and $1/2\le \eta\le 1.$ We denote by $\mathfrak{E}(\epsilon_0,N)$ the class of smooth functions defined on $I^{d-1}\times I$ which satisfy
\[\| \psi -\psi_0-t\|_{C^N(I^{d-1}\times I)}\le \epsilon_0,\]
where $\psi_0(\zeta)=|\zeta|^2/2$ for $\zeta\in I^{d-1}.$ We now recall the following from \cite{lee1}.

\begin{prop}\cite[Proposition 3.2]{lee1}\label{prop:lee} Let $\phi $ be a smooth function supported in $[-1,1]$. If $p> \min\{p_0(d),\frac{2(d+2)}d\}$ and $\epsilon_0$ is sufficiently small, then for $\epsilon>0$ there is a positive integer $M=M(\epsilon)$ such that 
\begin{equation}\label{eq:lee}\Big\|\Big(\int_{-1}^1\Big| \phi\Big(\frac{\eta(D,t)(D_d-\psi(D',t))}\delta\Big) f\Big|^2dt\Big)^{1/2}\Big\|_p \le B\delta^{-\frac{d-2}2+\frac dp-\epsilon}\|f\|_p
\end{equation}
holds uniformly for $\psi\in \mathfrak{E}(\epsilon_0,M)$ and $\eta\in \CE(M)$ whenever $\supp\widehat{f}\subset  [-1/2,1/2]^{d}$. Here, we denote by $m(D)f$ the multiplier operator given by $\CF(m(D)f)(\xi)=m(\xi)\widehat{f}(\xi)$ and also write $D=(D',D_d)$ where $D', D_d$ correspond to the frequency variables $\xi',\xi_d$, respectively, where $\xi=(\xi',\xi_d)\in \R^{d-1}\times\R.$ 
 
\end{prop}

It is not difficult  to see that the  constant $B$ in \eqref{eq:lee} only depends on the $C^N$-norm of $\phi$ for some $N$  large enough, hence one can find $N=N(\epsilon)$, $\epsilon>0$, such that \eqref{eq:lee} holds uniformly for $\psi\in \mathfrak{E}(\epsilon_0,N)$, $\eta\in \CE(N)$, and $\phi\in \CC_N([-1,1])$. In fact, $C^N$-norm is involved with kernel estimate which is needed for localization argument and $N$ can be taken to be as large as $\sim d$. 
 As mentioned  in Remark 3.3 in \cite{lee1}, Proposition \ref{prop:lee} implies  the following.
 
 \begin{lemma}\label{lem:lee'}  For any $\epsilon>0$, there is an $N$ such that   \eqref{del-square}  holds uniformly for all $\phi\in\CC_N([-1,1])$, if $p>p_s(d)=\min\{ p_0(d),\frac{2(d+2)}d\}$.
 \end{lemma} 
 
 In fact, let $\epsilon_0>0$ be sufficiently small. By finite decompositions, rotation, scaling, and change of variable, it suffices to prove that 
\[ \Big\|\Big(\int_{I_{\epsilon_0}}\Big|\phi\Big(\frac{t^2-|D|^2}\delta\Big) f\Big|^2dt\Big)^{1/2}\Big\|_p\le C\delta^{\frac{2-d}2+\frac dp-\epsilon}\|f\|_p,\quad \forall \supp \widehat{f}\subset B(-e_d,c\epsilon_0^2),\]
for $I_{\epsilon_0}=(1-{\epsilon_0}^2,1+{\epsilon_0}^2)$ and $e_d=(0,\cdots,0,1).$ 
Note that $t^2-|\xi|^2= -(\tau +\sqrt{t^2-|\zeta|^2})(\tau-\sqrt{t^2-|\zeta|^2})$ for $\xi=(\zeta,\tau)\in B(-e_d,c{\epsilon_0}^2)$. Here 
$(\zeta,\tau)=\mathbb R^{d-1}\times \mathbb R$. 
Then the simple change of variables in Remark 3.3 in \cite{lee1} transforms $\phi(\frac{t^2-|\xi|^2}\delta)$ to $\phi(\frac{2\eta(\xi,t)(\tau-\psi)}{\epsilon^{-2}_0\delta})$ for some $\psi\in \mathfrak{E}(C\epsilon^2_0,N)$ and $\eta\in \mathcal E(N)$. Applying Proposition \ref{prop:lee} we obtain Lemma \ref{lem:lee'}. 

Proposition \ref{prop:main} below follows from Lemma \ref{lem:lee'} and Lemma \ref{lem:sq1}.  

\begin{prop}\label{prop:main} Let  $0<\delta_0\ll 1$. Then, for $p>p_s(d)$ and  any $\epsilon>0$ there is $N=N(\epsilon)$ so that
\[ \|\mathfrak{D}^\phi_\delta f\|_p\le C \delta^{\frac{1-d}2+\frac dp -\epsilon}\|f\|_p\]
holds uniformly for $\phi\in\CC_N([-1,1])$ and $0<\delta\le \delta_0$.
\end{prop}

\subsection{$L^p-L^q$ estimates for  $\mathfrak{D}^\phi_\delta$} Note that the multiplier of $S_{\rho,\delta}^\phi$ in $\mathfrak{D}^\phi_\delta$ is supported in a $C\delta$-neighborhood of $\sqrt{\rho}$-sphere in $\R^d.$ Thus, by using Stein-Tomas theorem and well-known space localization argument we can obtain $L^p-L^q$ estimates for $\mathfrak{D}^\phi_\delta.$

\begin{prop}\label{prop:s-t}  Let $q\ge \frac{2(d+1)}{d-1}$ and $2\le p\le q$. Then for any $\epsilon>0$ there exists $N\in \mathbb{N}$ such that, for any $\phi\in \CC_{N}([-1,1])$ and $0<\delta\ll 1$, 
\[ 
\|\mathfrak{D}^\phi_\delta f\|_q\lesssim \delta^{\frac{1-d}2+\frac dp-\epsilon}\|f\|_p.
\]
{Here the implicit constant is independent of $\delta$ and $\phi$.}
\end{prop}

Interpolation between these estimates and those in Proposition 2.7 give additional estimates. The loss $\delta^{-\epsilon}$ can be removed by using better localization argument. See, for example, \cite{LRV}. But we don't attempt to this here.

Before proving Proposition \ref{prop:s-t}, we recall Stein-Tomas theorem (\cite{stein1}): For any $q \ge \frac{2(d+1)}{d-1}$, there is $C=C(p, d)>0$ such that
\[ \|\widehat{fd\sigma}\|_{L^q(\BBR^d)} \le C \|f\|_{L^2(\BBS^{d-1})},\]
where $\BBS^{d-1} $ is the  unit sphere in $\R^d$ and $d\sigma$ is the induced Lebesgue measure on $\BBS^{d-1}$. Using the polar coordinate, Stein-Tomas theorem, and mean-value theorem it is easy to see that, for $q \ge \frac{2(d+1)}{d-1}$,
\begin{equation}\label{eq:s-t} \|S_{\rho,\delta}^\phi f\|_q \lesssim \delta^{1/2}\|f\|_2\quad \mbox{ for }1/2\le \rho\le 2 .\end{equation}

\begin{proof}[Proof of Proposition \ref{prop:s-t}]
We denote by $K_\rho$ the kernel of $S_{\rho,\delta}^\phi$ in short. By  Lemma \ref{lem:kernel}, we see that for $1/2\le \rho\le 1$, $|K_{\rho}(x)| \le C_N \delta(1+\delta|x|)^{-N}$.  Recall that  $C_N$ is independent of $\delta$ and the choice of $\phi\in \CC_N([-1,1])$. This means that the kernel $K_{\rho}$ is essentially supported in a ball  of radius $\sim \delta^{-1}$. This enable us to use spatial localization argument, which deduces $L^p$ estimates for $\mathfrak{D}^\phi_\delta$ from $L^2\to L^p$ bound.

Let $\epsilon'>0$. We first restrict $f$ into balls of radius $\delta^{-1-\epsilon'}$: set $f_l = f\chi_{B(l, 3 \delta^{-1-\epsilon'})},~ l\in \delta^{-1}\BBZ^d$. For $x \in B(l,\delta^{-1-\epsilon'})$, we see that 
\begin{eqnarray*}
|S_{\rho,\delta}^\phi(f-f_l)(x)| \le C_N\delta  \int_{|y|\ge 2\delta^{-1-\epsilon'}}(1+\delta|y|)^{-d-1}|f(x-y)| dy \le E\ast |f|(x),
\end{eqnarray*}
where $E(x) = C_N\delta^{\epsilon'K}(1+\delta|x|)^{-d-1}$ and $K= N-d-1$. Since $q>2,$ we have
\begin{align}
& \quad \|(\sum_{\rho \in \delta\BBZ\cap [1/2,1]} |S_{\rho,\delta}^\phi f |^2)^{1/2}\|^q_q \nonumber 
\\
	\lesssim&   \sum_{l\in \delta^{-1}\BBZ^d} \int_{B(l,\delta^{-1-\epsilon'})} \Big[(\sum_{\rho} |S_{\rho,\delta}^\phi f_l |^2)^{q/2} +(\sum_{\rho} |S_{\rho,\delta}^\phi(f-f_l) |^2)^{q/2}\Big] dx \nonumber  
	\\
	\lesssim&\,\,  \delta^{-C\epsilon'} \sum_l(\sum_{\rho \in \delta\BBZ\cap [1/2,1]} \| S_{\rho,\delta}^\phi f_l \|_q^2)^{q/2} 
+  \delta^{-C\epsilon'}\int_{\R^d} 	(\sum_{\rho \in \delta\BBZ\cap [1/2,1]} (E\ast|f|)^2)^{q/2} dx \label{eq: square}.
\end{align}
Here the implicit constant depends only on $d.$ Notice that $S_{\rho,\delta}^\phi f_l =S_{\rho,\delta}^\phi\CP_{\rho}f_l$, where $\CP_{\rho}h$ is defined by 
\begin{equation}\label{eq:pro2}
 \widehat{\CP_{\rho}h}(\xi) =\chi_{\rho} (\xi) \widehat{h}(\xi) 
 \end{equation}
and $\chi_{\rho}$ is a characteristic function of $\Delta_\rho:=\{ \xi \in \BBR^d: |\xi|^2 \in [\rho-\delta,\rho+\delta] \}$. Since $\Delta_\rho$ are overlapping at most twice,  $  \sum_{\rho\in \delta\BBZ\cap [1/2,1]} \|\CP_{\rho}f_l\|^2_2\le 2\|f_l\|_2^2. $
By using this, the first term of \eqref{eq: square} is bounded by $(C \delta^{-(\frac{d-1}{2}-\frac{d}{p})-\epsilon'(\frac d2-\frac dp+C)})^q\|f\|_p^q$ because of \eqref{eq:s-t}, $l^p\subset l^q$, and $p\ge2.$
Since $0<\delta <1$ and $p\le q$, the second term of \eqref{eq: square} is bounded by
\begin{eqnarray*}
 \delta^{-C\epsilon'} \Big(\sum_{\rho\in \delta\BBZ \cap [1/2,1]} \|E\ast |f| \|^2_{q}\Big)^{q/2}\le C_K\delta^{-C\epsilon'-q/2}\delta^{\epsilon'Kq(1-1/2)}\|f\|_p^q\lesssim \|f\|_p^q.
\end{eqnarray*}
if $K$ is sufficiently large (i.e., $N$ is large enough).   
Thus, taking  $\epsilon'=\epsilon/C$ for some large $C$, we get the desired inequality. 
\end{proof}

\end{section}

\begin{section}{Reduction; decomposition lemma}\label{sec:re}
In this section, we will  break the operator $\CB^\alpha$   so that our problem  is reduced to obtaining bounds for  a simpler  bilinear operator  which is given by products of $S_{\rho,\delta}^\phi$ with different $\rho$ which is defined by \eqref{eq:S}. This reduction enables us to draw connection to the square function estimate. To do this, we first consider an auxiliary  bilinear operators $\widetilde{\CB}_{\delta}$, $0<\delta \ll 1$ which is given  by dyadically decomposing the multiplier of $\CB^\alpha$ away from its singularity $\{(\xi,\eta): |\xi|^2+|\eta|^2 =1\}$.

Let us denote by $\mathcal D$ the set of positive dyadic numbers, that is to say $\mathcal D=\{2^k: k\in \mathbb Z\}$. Fix $\alpha>0$ and  let $\psi$ be a function $ \in C^{\infty}_c (1/2,2)$ satisfying 
$  \sum_{\delta  \in  \mathcal D  } \delta^\alpha \psi(t/\delta)=t^\alpha, \   t>0.$
Then we may write 
\[  (1-t)_+^\alpha = \sum_{\delta  \in  \mathcal D: \delta\le 2^{-1}  } \delta^\alpha \psi\Big(\frac{1-t}{\delta}\Big)  +\psi_0(t),    \   t\in [0,1). \] 
where $\psi_0$ is a smooth function supported in $[0, 3/4]$.   
Using this we decompose $\mathcal B^\alpha$ so that 
\Be 
\label{del-decomp}  
 \CB^\alpha=  \sum_{\delta  \in  \mathcal D: \delta\le 2^{-1}  }  \delta^\alpha \widetilde \CB_\delta +  \widetilde\CB_0,
\Ee
where 
\Be\label{b-delta}  \widetilde \CB_\delta (f,g)(x)  = \int_{\mathbb{R}^d}\int_{\mathbb{R}^d} e^{2\pi i x\cdot(\xi+\eta)} \psi\Big(\frac{1-|\xi|^2-|\eta|^2 }{\delta}\Big) ~ \widehat{f}(\xi)~\widehat{g}(\eta)d\xi d\eta 
\Ee
and  $\widetilde{\CB}_0$ is similarly defined by $\psi_0$. Since $\psi_0\in C_c([0, 3/4])$, it is easy to see that 
\[ \|   \widetilde\CB_0(f,g)\|_{r}\le C\|f\|_p\|g\|_q\]  
whenever $1/r\le 1/p+1/q. $ Thus,  in order to show  \eqref{eq:b} for $\alpha> \kappa$ it is sufficient to show that, for any $\epsilon>0$, there exits $C_\epsilon$ such that  
 \[ \|   \widetilde\CB_\delta(f,g)\|_{r}\le C_\epsilon \delta^{{-\kappa}-\epsilon} \|f\|_p\|g\|_q. \]

For  $\phi_1, \phi_2$  be smooth functions supported in $[-1,1]$, and $\varrho\in [1/2, 2]$, we  define the bilinear operators $\CB_{\delta, \varrho}^{\phi_1, \phi_2}$  
by setting 
\begin{eqnarray}\label{op:B}
  \CB_{\delta, \varrho}^{\phi_1, \phi_2}(f,g)(x)
 &:=&\sum_{\rho\in\delta\mathbb Z\cap[0,1]} S_{\rho,\delta}^{\phi_1} f(x) S_{\varrho-\rho,\delta}^{\phi_2}g(x)
 \end{eqnarray}
 Thanks to the above argument  and Lemma \ref{Lemm:reduction} below,   instead of  $\mathcal\CB_\delta$ it suffices to obtain bounds for $\CB_{\delta,\varrho}^{\phi_1, \phi_2}$ of which 
product structure makes the problem easier.

\begin{lemma}\label{Lemm:reduction} Let  $\kappa\ge -1$,  $0<\delta_0\ll 1$ and $1\le p,q,r\le \infty$ satisfy $1/p+1/q \ge 1/r.$    Suppose that,  for any $0<\delta \ll \delta_0$ and $\varrho\in [1/2, 2]$,  
\begin{equation}\label{eq:assumption}
 \|\CB_{\delta,\varrho}^{\phi_1, \phi_2}\|_{L^p\times L^q \to L^r} \le A \delta^{-\kappa}
 \end{equation}
 holds uniformly with  $A>0$ independent of {$\delta,\varrho$, and} $\phi_1, \phi_2$,  whenever  $\phi_1, \phi_2\in \CC_N([-1,1])$  for some $N$. Then, 
  for any $\epsilon>0$ there exists a constant  $A_{\epsilon}$, independent of $\delta$,  such that 
  \[ \|\widetilde{\CB}_{\delta}\|_{L^p\times L^q \to L^r} \le A_{\epsilon} \delta^{-\kappa-\epsilon(1+\kappa)}.\]
\end{lemma}

It is not difficult to see that \eqref{eq:assumption} does not hold for $\kappa<-1$. In fact, let $f$, $g$ be smooth functions such that  $\supp \widehat f,$ $ \supp \widehat g\subset  B(0, 4)$ and 
$\widehat f= \widehat g=1$   on  $B(0, 3)$ and $\phi=\phi_1=\phi_2$ be nontrivial nonnegative functions with  $\supp \phi\subset [-1,1]$. Then,   it is easy to see that $|\CB_{\delta,\varrho}^{\phi_1, \phi_2}(f,g)(x)|\gtrsim \delta$  if $|x|\le c$ with sufficiently small $c>0$. 
Thus $\|\CB_{\delta,\varrho}^{\phi_1, \phi_2}(f,g)\|_r\gtrsim \delta$ while $\|f\|_p, \|g\|_q\lesssim 1$. This implies $\|\CB_{\delta,\varrho}^{\phi_1, \phi_2}\|_{L^p\times L^q \to L^r} \gtrsim  \delta$.

\begin{remark}\label{re:trivial}Using Lemma \ref{Lemm:reduction} and the trivial $L^2$-estimate for $S_{\rho,\delta}^\phi$, we can easily recover the boundedness of  $\CB^\alpha$ from $L^2(\R^d)\times L^2(\R^d)$ into $L^1(\R^d)$ for $\alpha>0$ (Proposition  \ref{thm:bg}).
Applying Schwarz's inequality and Plancherel's theorem, we have 
\[\|\CB_{\delta,\varrho}^{\phi_1, \phi_2}(f,g)\|_1\le  ( \sum_{\rho\in \delta\Z\cap [0,1]} \| S_{\rho,\delta}^{\phi_1}f\|_2^2)^{1/2}( \sum_{\rho\in \delta\Z\cap [0,1]} \| S_{\varrho-\rho,\delta}^{\phi_2} g\|_2^2)^{1/2}.
\]
Since $\sum_{\rho\in \delta\Z\cap [0,1]} \|S_{\rho,\delta}^{\phi_1} f\|_{2}^2 \lesssim \|f\|_2^2$, $\sum_{\rho\in \delta\Z\cap [0,1]} \|S_{\varrho-\rho,\delta}^{\phi_2} g\|_{2}^2 \lesssim \|g\|_2^2$,  it follows from the above that   $\|\CB_{\delta,\varrho}^{\phi_1,\phi_2}\|_{L^2\times L^2 \to L^1} \lesssim 1$. By Lemma \ref{Lemm:reduction} $\|\widetilde{\CB}_{\delta}\|_{L^2\times L^2 \to L^1} \le A_{\epsilon} \delta^{-\epsilon} $ and, hence, from \eqref{del-decomp} we see that $\CB^\alpha$ is bounded from $L^2(\R^d)\times L^2(\R^d)$ into $L^1(\R^d)$ for all $\alpha>0.$
\end{remark}

 \begin{proof}[Proof of Lemma \ref{Lemm:reduction}]
Let $\varphi \in C^{\infty}_{c}([-1,1])$ satisfy
\begin{equation}\label{eq: equal}
 \sum_{k\in \BBZ}\varphi(t+k) = 1, \quad t\in \BBR.
 \end{equation}
 Using this, we will decompose the multiplier of $\widetilde{B}_\delta$ into sum of  multipliers which are given by (tensor) product of two multipliers supported in thin annuli. More precisely, we fix small $\epsilon>0$ and $0<\delta\le\delta_0$, and set $\tilde{\delta}=\delta^{1+\epsilon}<\delta$. Then
 \begin{align*}
\psi\Big(\frac{1-|\xi|^2-|\eta|^2}\delta\Big)
 = \sum_{\rho\in \tilde{\delta}\BBZ\cap[0,1]}  \sum_{\varrho\in \tilde{\delta}\BBZ}\varphi\Big(\frac{\rho-|\xi|^2}{\tilde{\delta}}\Big)\varphi\Big(\frac{\varrho-\rho-|\eta|^2}{\tilde{\delta}}\Big)\psi\Big(\frac{1-|\xi|^2-|\eta|^2}\delta\Big).
 \end{align*}
Note that $\varphi(\frac{\rho-|\xi|^2}{\tilde{\delta}}) \psi(\frac{1-|\xi|^2-|\eta|^2}\delta)\neq 0$ implies $1-3\delta \le |\eta|^2+\rho\le 1+\delta,$ since $\supp \varphi\subset[-1,1]$ and $\supp \psi\subset[1/2,2]$. The summands  vanish if we take the sum over $\varrho\in \tilde{\delta}\BBZ\cap(\R\setminus [1-4\delta,1+2\delta])$. Thus we can write
\begin{align}
\label{op} \widetilde{\CB}_{\delta}&(f,g)(x)= \sum_{\varrho\in\tilde{\delta}\BBZ\cap[1-4\delta, 1+2\delta]} \sum_{\rho \in \tilde{\delta}\BBZ\cap[0,1]} \iint_{\BBR^d\times\BBR^d}   e^{2\pi i x\cdot(\xi +\eta)}  \\&\times \psi\Big(\frac{1-|\xi|^2-|\eta|^2}{\delta}\Big) 
 \varphi\Big(\frac{\rho -|\xi|^2}{\tilde{\delta}}\Big)\varphi\Big(\frac{\varrho- \rho -|\eta|^2}{\tilde{\delta}}\Big)\widehat{f}(\xi)\widehat{g}(\eta) d\xi d\eta. \nonumber
\end{align}

Let  $N$ be a constant to be chosen later.  By Taylor's theorem we may write  
\[\begin{aligned}e^{2\pi i (\frac{\varrho-|\xi|^2-|\eta|^2}{\delta})\tau} =\sum_{0\le\beta+\gamma\le N}&C_{\beta,\gamma}\tau^{\beta+\gamma}\Big(\frac{\rho-|\xi|^2}{\delta}\Big)^{\beta}\Big(\frac{\varrho-\rho-|\eta|^2}{\delta}\Big)^{\gamma} \\ &+E\Big(2\pi i (\frac{\varrho-|\xi|^2-|\eta|^2}{\delta})\tau\Big),
\end{aligned}
\]
for any  $\rho \in \tilde{\delta}\BBZ\cap[0,1]$ and the remainder $E$ satisfies, for $0\le k\le N$,
\begin{equation}\label{eq:E}
 | E^{(k)}(t)|\le C_k |t|^{N-k}.
 \end{equation}
Using inversion,  for any $\varrho$ we have 
\Be\label{expand} \psi\Big(\frac{1-|\xi|^2-|\eta|^2}{\delta}\Big) =\int_{\BBR} \widehat{\psi}(\tau)e^{2\pi i \frac{1-\varrho}{\delta}\tau}e^{2\pi i (\frac{\varrho-|\xi|^2-|\eta|^2}{\delta})\tau}d\tau. \Ee

For  $0\le \beta \le N$ we set $\phi_{\beta}(t) =t^{\beta}\varphi(t)\in C_0^{\infty}(-1,1)$ and also set 
\[ \widehat \psi_\varrho=  \widehat{\psi}(\tau)e^{2\pi i \frac{1-\varrho}{\delta}\tau}.  \] Then,  putting  the 
above  in the right hand side of \eqref{expand}, we have 
\Be\label{op2}\begin{aligned}
&\qquad\quad \psi\Big(\frac{1-|\xi|^2-|\eta|^2}{\delta}\Big)=\sum_{0\le\beta+\gamma\le N}C_{\beta,\gamma}\Big(\int \widehat{\psi }_\varrho (\tau)\tau^{\beta+\gamma} d\tau\Big) \\ &\times  \phi_\beta\Big(\frac{\rho-|\xi|^2}{\delta}\Big)\phi_\gamma\Big(\frac{\varrho-\rho-|\eta|^2}{\delta}\Big)+ \int  \widehat{\psi }_\varrho (\tau) E\Big(2\pi i (\frac{\varrho-|\xi|^2-|\eta|^2}{\delta})\tau\Big)  \,d \tau 
\end{aligned}
\Ee
For each $0\le \beta, \gamma\le N$, we set  
\[I^\varrho_{\beta,\gamma}(f,g)=\Big(\int \widehat{\psi }_\varrho (\tau)\tau^{\beta+\gamma} d\tau\Big)\times\Big(\sum_{\rho\in\tilde{\delta}\Z\cap[0,1]}S_{\rho,\tilde{\delta}}^{\phi_\beta}f(x)S_{\varrho-\rho, \tilde{\delta}}^{\phi_\gamma}g(x)\Big).\]
Inserting \eqref{op2} in \eqref{op},  we express $\widetilde{\CB}_{\delta}$ as a sum of bilinear operators which are given by products of $S_{\rho,\tilde{\delta}}^{\phi_\beta}$:
\Be
\widetilde{\CB}_{\delta}(f,g)=\sum_{\varrho\in\tilde{\delta}\BBZ\cap[1-4\delta,1+2\delta]}\Big[ \sum_{0\le \beta+\gamma\le N}C_{\beta,\gamma}  \delta^{\epsilon(\beta+\gamma)}
 I^\varrho_{\beta,\gamma}  (f,g) +I^\varrho_E  (f,g) \Big],\label{eq:decom}
\Ee
where 
\[ I^\varrho_E(f,g)=
\sum_{\rho}\int \widehat{\psi_\varrho}(\tau) \iint e^{2\pi i x\cdot(\xi +\eta)} E_{\delta,\tilde \delta}(\xi,\eta,\rho, \varrho,\tau) \widehat{f}(\xi)\widehat{g}(\eta)d\xi d\eta d\tau.\]
and 
\[E_{\delta,\tilde \delta}(\xi,\eta,\rho, \varrho,\tau) = E\Big(2\pi i (\frac{\varrho-|\xi|^2-|\eta|^2}{\delta})\tau\Big) \varphi\big(\frac{\rho -|\xi|^2}{\tilde{\delta}}\big)\varphi\big(\frac{\varrho- \rho -|\eta|^2}{\tilde{\delta}}\big).\]

 Set $M= \max\{\|\varphi_\beta\|_{C^N([-1,1])}: 0\le\beta \le N\}$. 
 Then $M^{-1}\varphi_\beta \in \CC_N([-1,1])$ for all $0\le\beta \le N$. Thus, from  the assumption \eqref{eq:assumption} we have that, for each $I^\varrho_{\beta,\gamma}$, there exists a constant  $A$ such that
\Be\label{I-rho}
\|I^\varrho_{\beta,\gamma}(f,g)\|_r\le AM\tilde{\delta}^{-\kappa}\|f\|_p\|g\|_q,
\Ee
since $\psi$ is a Schwartz function.

Now, in order to complete the proof,  it is sufficient to show 
\Be \| I^\varrho_E(f,g)\|_r\le   A\delta^{-\kappa(1+\epsilon)}\|f\|_p\|g\|_q . \Ee  For the purpose, we use Lemma \ref{Lemm:I_E} below, which is a simple consequence of the bilinear interpolation.
\begin{lemma}\label{Lemm:I_E}
Let $0<\delta <1$ and $\tau \in \BBR$. Fix a large integer $N> 2d$.
Suppose that $m_{\delta,\tau} \in C^{\infty}_0(\BBR^d\times\BBR^d)$ is a smooth function supported in the cube $[-2,2]^{2d}$ in $ \BBR^{2d}$, and suppose that $m_{\delta,\tau}$ satisfies
\[|\partial_{\xi}^{\alpha}\partial_{\eta}^{\beta} m_{\delta,\tau}(\xi,\eta)| \le C_{\alpha,\beta}(1+|\tau|)^{N}\delta^{-|\alpha|-|\beta|} \]
 for all multi-indices $\alpha$, $\beta$ with $|\alpha|+|\beta|\le N.$ Let $T_{\delta,\tau}$ be defined by
 \[ T_{\delta,\tau}(f,g) = \iint_{\BBR^d\times\BBR^d} e^{2\pi i x\cdot(\xi+\eta)} m_{\delta,\tau}(\xi,\eta) \widehat{f}(\xi)\widehat{g}(\eta)d\xi d\eta,\quad  f,g \in \CS(\BBR^d)\]
Then, for $p,q,r \in [1,\infty]$ and
$ \frac{1}{p}+\frac{1}{q} \ge \frac{1}{r}$, we have
\[ \|T_{\delta,\tau}(f,g)\|_r\le C(1+|\tau|)^N\delta^{-d(2-\frac1p-\frac1q)}\|f\|_p\|g\|_q.\]
\end{lemma} 
\begin{proof} By definition, we can write
\[ T_{\delta,\tau}(f,g)(x) = \iint \widehat{m_{\delta,\tau}}(y-x,z-x) f(y)g(z) dydz.\]
Applying usual integration by parts, we have
$ |\widehat{m_{\delta,\tau}}(y,z)| \le C_K(1+ |\tau|)^N(1+\delta|y|)^{-N_1}(1+\delta|z|)^{-N_2}$ 
for all $N_1+N_2\le  N$.   Since $N$ is an integer bigger than $2d$, in particular, we have 
\[ |\widehat{m_{\delta,\tau}}(y,z)| \le C (1+ |\tau|)^N(1+\delta|y|)^{-d-\frac12}(1+\delta|z|)^{-d-\frac12}.\] Thus, for any $p,q\ge 1$, 
\[ \|T_{\delta,\tau}(f,g)\|_\infty \le C (1+|\tau|)^N\delta^{-d(2-\frac1p-\frac1q)}\|f\|_p\|g\|_q\]
On the other hand, by Fubini's theorem
we have
\[ \|T_{\delta,\tau}(f,g)\|_1 \le C (1+|\tau|)^N\delta^{-d(2-\frac1p-\frac1q)}\|f\|_p\|g\|_q\]
for any $p,q \ge 1$ with $\frac1p+\frac1q\ge 1.$
The bilinear interpolation between these two estimates gives all the desired estimates.
\end{proof}

\smallskip

In order to apply Lemma \ref{Lemm:I_E} to $I^\varrho_E$, we define a function $m$ on $\BBR^d\times\BBR^d\times\BBR$ by  
\[ m(\xi,\eta,\tau) = \delta^{-\epsilon N}  E_{\delta,\tilde \delta}(\xi,\eta,\rho, \varrho,\tau) \]
for some $\rho \in \tilde{\delta}\BBZ\cap[0,1]$, then $m(\cdot,\cdot,\tau)$ satisfies all properties of the function $m_{\tilde{\delta},\tau}$ in Lemma \ref{Lemm:I_E}, because of \eqref{eq:E}. More precisely, using \eqref{eq:E} we see that 
\begin{align*}
|m(\xi,\eta,\tau)|&\le C_0|\tau|^{N}\Big({\Big|}\frac{\rho-|\xi|^2}{\tilde{\delta}}{\Big|}+ {\Big|}\frac{\varrho-\rho -|\eta|^2}{\tilde{\delta}}{\Big|}\Big)^{N}\Big|\varphi\Big(\frac{\rho-|\xi|^2}{\tilde{\delta}}\Big)\varphi\Big(\frac{\varrho-\rho-|\eta|^2}{\tilde{\delta}}\Big)\Big|\\
&\le C_0(1+|\tau|)^N 2^N\|\varphi\|^2_{\infty},
\end{align*}
and similarly, by direct differentiation and using \eqref{eq:E}  we also have, for $\beta,\gamma$ with $|\beta|+|\gamma|\le N$, 
\[
|\partial^{\beta}_{\xi}\partial^{\gamma}_{\eta}m(\xi,\eta,\tau)|\le C(1+|\tau|)^N(\tilde{\delta})^{-\beta-\gamma}.
\]
Here $C$ depends only on $C_k$ in \eqref{eq:E} and $M$.  
We note that  $I^\varrho_E(f,g)$ is expressed by
\begin{eqnarray*}
I^\varrho_E (f,g)(x)=\delta^{\epsilon N}\sum_{\rho\in \tilde{\delta}\BBZ\cap[0,1]}\int e^{2\pi i \frac{\rho-\varrho}{\delta}}\widehat{\psi}(\tau) T_{\tilde{\delta},\tau}(f,g)(x) d\tau,
\end{eqnarray*}
where $T_{\tilde{\delta},\tau}$ is defined as in Lemma \ref{Lemm:I_E}. Thus, by Lemma \ref{Lemm:I_E} and Minkowski's inequality 
we obtain
\begin{align*}
\|I^\varrho_E(f,g)\|_r  
&\le C \delta^{\epsilon N}\tilde{\delta}^{-2d}\|f\|_p\|g\|_q \sum_{\rho}\int |\widehat{\psi}(\tau)|(1+|\tau| )^Nd\tau \\\ 
&
\lesssim \delta^{\epsilon N}\tilde{\delta}^{-2d-1}\|f\|_p\|g\|_q,
\end{align*}
provided that $\frac{1}{p}+\frac{1}{q} \ge\frac{1}{r}$. Thus, combining this estimate, \eqref{eq:decom} and \eqref{I-rho}, we obtain
\[
\|\widetilde{\CB}_{\delta}(f,g)\|_r
\lesssim\sum_{\varrho\in\tilde{\delta}\BBZ\cap[1-4\delta, 1+2\delta]}\Big[ \sum_{0\le\beta+\gamma\le N} |C_{\beta,\gamma}|\delta^{\epsilon(\beta+\gamma)}\tilde{\delta}^{-\kappa} +\delta^{\epsilon( N-2d-1)-2d-1} \Big] \|f\|_p\|g\|_q.
\]
Therefore, choosing sufficiently large $N$, we have 
\[
\|\widetilde{\CB}_{\delta}(f,g)\|_r \lesssim  \delta^{-\kappa-\epsilon(1+\kappa)} \|f\|_p\|g\|_q. 
\]
 Here the implicit constant is independent of $\delta.$
\end{proof}

\end{section}

\begin{section}{Boundedness of bilinear Bochner-Riesz operators} 
In this section  we prove Theorem \ref{thm:main} and also obtain results for the  sub-critical case $\frac1p+\frac1q>\frac1r$ mostly relying on Stein-Tomas's theorem.  In addition, we find a necessary condition for $\CB^\alpha$ by using duality and asymptotic behavior of localized kernel.

\subsection{Proof of Theorem \ref{thm:main}} To verify Theorem \ref{thm:main}, by \eqref{del-decomp} and Lemma \ref{Lemm:reduction}  it is enough to show Proposition \ref{prop:TS1} below by using the argument in Section \ref{sec:re}.
\begin{prop}\label{prop:TS1}
 Let $d\ge 2$, $0<\delta\le \delta_0\ll 1$, $2\le p,q\le\infty$ and $1/r=1/p+1/q.$ 
 Suppose that for $p\ge p_\circ$ the estimate \eqref{del-square} holds with $C$ independent of {$\delta$ and }$\phi$ whenever $\phi\in \CC_N([-1,1])$ for some $N$. Then, for any $\epsilon>0$, there exist $N=N(\epsilon)$  and $C_\epsilon$ such that for $\varrho\in [1/2,2]$
\begin{equation}\label{eq:TS}
  \|\CB_{\delta,\varrho}^{\phi_1, \phi_2}(f,g)\|_{L^r(\R^d)} \le C_{\epsilon}\delta^{-\alpha_{\frac1{p_\circ}}(1/p,1/q)-\epsilon } \|f\|_{L^p(\R^d)}\|g\|_{L^q(\R^d)}
\end{equation}
 holds uniformly provided that   $\phi_1$ and $\phi_2$ in $\CC_N([-1,1])$.
 \end{prop}

\begin{proof} In view of the interpolation it suffices to prove \eqref{eq:TS} for critical pairs of exponents $(1/p,1/q)$ which are in $\Delta_1=\Delta_1(\frac1{p_o})$,  $\{(1/2,1/p_\circ)\}$, $\{(1/p_\circ,1/2)\}$,  $\{(1/2,0)\}$, $\{(0,1/2)\}$, and $\{(1/2,1/2)\}.$

We first consider the  case $(1/p,1/q)\in \Delta_1$. We fix $(1/p,1/q)\in \Delta_1$, i,e, $p,q\ge p_\circ$. Then it is sufficient to show that for any $\epsilon>0$ there is an $N$ such that 
\[ \|\CB_{\delta,\varrho}^{\phi_1,\phi_2}(f,g)\|_r\lesssim \delta^{-\beta_*(\frac1p)-\beta_*(\frac1q)-\epsilon}\|f\|_p\|g\|_q,\]
where $\CB_{\delta,\varrho}^{\phi_1,\phi_2}$ is associated with $\phi_j\in \CC_N([-1,1])$ and the implicit constant is independent of the choice of $\phi_j$'s and $\delta,\varrho$.
Recall that 
\[\CB_{\delta, \varrho}^{\phi_1,\phi_2}(f,g) =\sum_{\rho\in\delta\BBZ\cap[0,1]} (S^{\phi_1}_{\rho,\delta}f)( S^{\phi_2}_{\varrho-\rho,\delta}g).\]
By Schwarz's inequality, for any $x\in \R^d$
\Be \label{cauchy-schwarz} 
|\CB_{\delta, \varrho}^{\phi_1,\phi_2}(f,g)(x)| \le\Big(\sum_{\rho\in\delta\Z\cap[0,1]} |S^{\phi_1}_{\rho,\delta}f(x)|^2\Big)^{1/2}\Big(\sum_{\rho\in\delta\Z\cap[0,1]}| S^{\phi_2}_{\varrho-\rho,\delta}g(x)|^2\Big)^{1/2}.
\Ee
In this case we only deal with the triple pair of exponents $(p,q,r)$ satisfying H\"older's relation. Hence, by H\"older's inequality, it suffices to show that 
\begin{equation}\label{eq:tem} \Big\|\Big(\sum_{\rho\in\delta\Z\cap[0,1]} |S^{\phi}_{\rho,\delta}f(x)|^2\Big)^{1/2}\Big\|_p\le C\delta^{-\beta_*(1/p)-\epsilon}\|f\|_p
\end{equation}
for $p\ge p_\circ.$ Indeed, since each $\varrho$ is small perturbation of $1$ and $q\ge p_\circ,$ the same argument which shows \eqref{eq:tem} implies the uniform bounds for $L^q$-estimate for $(\sum_{\rho\in \delta\Z\cap[0,1]}|S^{\phi}_{\varrho-\rho,\delta}g|^2)^{1/2}.$  
We now prove \eqref{eq:tem}. It is a consequence of the estimates for square functions $\mathfrak{D}^\phi_\delta$ in subsection \ref{sub:sq}. In fact, set $C_1=\delta_0^{-1}$ and decompose the interval $[C_1\delta,1]$ dyadically, i.e., we set
\begin{equation*}
 \bigcup_{k=0}^{k_o}\BI_k:= \bigcup_{k=0}^{k_o}[2^{-k-1},2^{-k}] \cap [C_1\delta,1]=[C_1\delta,1],  
\end{equation*}
where $k_o+1$ is the smallest integer satisfying $[2^{-k-1},2^{-k}] \cap [C_1\delta,1]=\emptyset$. By triangle inequality, we have
\Be\label{eq:tr} \begin{aligned}
\Big(\sum_{\rho\in\delta\Z\cap[0,1]} & |S^{\phi}_{\rho,\delta}f(x)|^2\Big)^{\frac12}
\\
\le &\Big(\sum_{\rho\in\delta\Z\cap[0,C_1\delta]} |S^{\phi}_{\rho,\delta}f(x)|^2\Big)^{\frac12}
+ \sum_{k=0}^{k_o}\Big(\sum_{\rho\in\delta\Z\cap\BI_k} |S^{\phi}_{\rho,\delta}f(x)|^2\Big)^{\frac12}.
\end{aligned}
\Ee
When $k=0$, Lemma \ref{lem:sq1} implies $\|(\sum_{\rho\in \delta\Z\cap \BI_0}|S^{\phi}_{\rho,\delta}f|^2)^{1/2}\|_p\le C_\epsilon \delta^{-\beta_*(1/p)-\epsilon}\|f\|_p$ holds uniformly for $\phi\in\CC_N([-1,1])$. By scaling $\xi\to 2^{-\frac k2}\xi$, it is easy to see that 
$  S^{\phi}_{\rho,\delta}f(x)=   S^{\phi}_{2^k\rho,  2^k\delta } \big(f( 2^{\frac k2}\cdot)\big)(2^{-\frac k2}x)$. 
Thus we have that 
\Be \label{scaled}
\|(\sum_{\rho\in \delta\Z\cap \BI_k}|S^{\phi}_{\rho,\delta}f|^2)^{1/2}\|_p
=2^{\frac {kd}{2p}} \|(\sum_{\rho\in 2^k\delta\Z\cap \BI_0}|S^{\phi}_{\rho,2^k\delta}  \big(f( 2^{\frac k2}\cdot)\big)|^2)^{1/2}\|_p.
\Ee
 Now, since  $2^k\delta\le \delta_0$,  using  Proposition \ref{prop:main} and recaling, we have   for $k\ge1$  
\[ \|(\sum_{\rho\in \delta\Z\cap \BI_k}|S^{\phi}_{\rho,\delta}f|^2)^{1/2}\|_p
\le C_\epsilon (2^k\delta)^{-\beta_*(1/p)-\epsilon}\|f\|_p. 
\]
 Since $\beta_*(1/p)>0$,  summing over $k$  we see that 
 \[  \sum_{k=0}^{k_o} \Big\| \Big(\sum_{\rho\in\delta\Z\cap\BI_k} |S^{\phi}_{\rho,\delta}f(x)|^2\Big)^{\frac12} \Big\|_{p} \le C_\epsilon \delta^{-\beta_*(1/p)-\epsilon} \|f\|_p\]
 For the first term in \eqref{eq:tr}, we recall from  Lemma \ref{lem:kernel} that  $\|K^{\phi}_{\rho,\delta}\|_1=O(1)$ for $\rho\lesssim \delta$. Thus, by Young's convolution inequality we have  that $\|S^{\phi}_{\rho,\delta}f\|_p\lesssim \|f\|_p$ for $\rho\lesssim \delta$.  There are only $O(1)$ many 
 $\rho\in\delta\Z\cap[0,C_1\delta]$. Thus it follows that 
 \[ \Big\| \Big(\sum_{\rho\in\delta\Z\cap[0,C_1\delta]} |S^{\phi}_{\rho,\delta}f(x)|^2\Big)^{\frac12} \Big\|_p\lesssim \|f\|_p.\] 
 Combining this with the above, we obtain \eqref{eq:tem}.

 We now consider the remaining  cases  $(p,q)=(2,2),  (2,\infty)$, $(\infty,2)$,  $(2, p_\circ)$, $(p_\circ, 2)$. The case  $(p,q)=(2,2)$ is already handled in Remark \ref{re:trivial}. It is sufficient to show \eqref{eq:TS} for  $(\infty,2)$,  $( p_\circ,2)$ since the other cases symmetrically follow by the same argument.   The proof of these two cases are rather straight forward.   From \eqref{cauchy-schwarz}, H\"older's inequality, \eqref{eq:tem}, and Plancherel's theorem  we have,  for $p\ge p_\circ$ and $\epsilon>0$, 
 \begin{align*}
 \|\CB^{\phi_1,\phi_2}_{\delta, \varrho}(f,g)\|_{r} 
 &  \le  \Big\|\Big(\sum_{\rho\in\delta\Z\cap[0,1]} |S^{\phi_1}_{\rho,\delta}f|^2\Big)^{1/2}\Big\|_p \Big\|\Big(\sum_{\rho\in\delta\Z\cap[0,1]} |S^{\phi_2}_{\varrho-\rho,\delta} g|^2\Big)^{1/2}\Big\|_2  
 \\ 
 &  \le C\delta^{-\beta_*(1/p)-\epsilon}\|f\|_p  \|g\|_2,
 \end{align*}
 where $1/r=1/p+1/2.$  This completes the proof. 
\end{proof}

\subsection{Sub-critical case:  $\frac1p+\frac1q \ge  \frac1r$ } In this subsection, we consider  $L^p\times L^q\to L^r$ boundedness  for the case $1/p+1/q\ge 1/r$.  For the rest of this section we set \[r_1=\frac{2(d+1)}{d-1},  \  r_2=\frac{2d}{d-2}.\]  The following is the main result of this section. 

\begin{theorem}\label{thm:sub}
Let $d\ge 2$, $2\le p,  q\le  \infty $ and $r\ge \frac{d+1}{d-1}$. If $1/p+1/q\ge 1/r$, then 
\begin{equation*}\label{eq:p'}
\|\CB^\alpha(f,g)\|_{L^{r}(\R^d)}\le C\|f\|_{L^p(\R^d)}\|g\|_{L^q(\R^d)}
\end{equation*}
holds for $\alpha>\gamma(p,q,r)$, where $\gamma(p,q,r)$ is defined as follows;
\[  \gamma(p,q,r)= \begin{cases} 
   \beta_*(1/p)+\beta_*(1/q) ,  & \text{ if }  \frac1r\le \frac1{r_1}+\frac 1{r_2},  \\
    \beta_*(1/p)+\beta_*(1/q)-\frac{d^2-d-1}{2(d+1)}+\frac d{2r},    &\text{ if }  \frac1{r_1}+\frac 1{r_2} \le \frac1r\le \frac2{r_1}.\\
\end{cases} 
\]
\end{theorem}

Further estimates are possible if we interpolate the estimates in the above with those in Theorem \ref{thm:main}. 

Recall \eqref{eq1} and  note that  the operator $\CB^\alpha$ is well-defined for $\alpha>-1.$  For $\alpha\le -1$, $\CB^\alpha(f,g)/\Gamma(\alpha+1)$ is defined  by  analytic continuation. $L^p$-$L^q$ estimates for the classical Bochner-Riesz operator of negative order have been studied by several authors \cite{b,Bo,So,c-k-l-s} and its connection to the Bochner-Riesz conjecture is now well understood. It also seems to be  an interesting problem  to characterize $L^p\times L^q\to L^r$ boundedness of $\CB^\alpha$ of negative order, but  such attempt might be premature in view of current state of art.

We deduce the estimates in Theorem  \ref{thm:sub}  from easier $L^2\times L^2\to L^r$ {estimates}. For the purpose we make use of the following localization lemma. 

 \begin{lemma}\label{lem:local} Let $1\le p,q,r, p_0, q_0, r_0\le \infty$ satisfy $1/p+1/q\ge 1/r$  and $p_0\le p, q_0\le q$,  $r\le r_0$, and let $\varrho\in [1/2,2]$. Suppose $\|\CB_{\delta, \varrho}^{\phi_1, \phi_2}(f,g)\|_{r_0} \le C  \delta^B \|f\|_{p_0}\|g\|_{q_0}$  holds uniformly  provided that   $\phi_1$ and $\phi_2$ in $\CC_N([-1,1])$, then  for any $\epsilon>0$, there are constants $C_\epsilon$ and $N'$, such that 
 \Be \label{local-local}
 \|\CB_{\delta, \varrho}^{\phi_1, \phi_2}(f,g)\|_{r} \lesssim   \delta^B  \delta^{d(\frac1p+\frac1q-\frac1r- \frac1{p_0} -\frac1{q_0}+\frac1{r_0})-\epsilon}\|f\|_{p}\|g\|_{q}
 \Ee
 holds uinformly whenever  $\phi_1$ and $\phi_2$ in $\CC_{N'}([-1,1])$.
 \end{lemma}

By further refinement of the argument below it is possible to remove $\epsilon>0$.   This lemma can be obtained by adapting the localization argument used for the proof of Proposition \ref{prop:s-t}. Hence, we shall be brief.  

\begin{proof}  
Let $\epsilon'>0$. As in the proof of Proposition \ref{prop:s-t}, we localize $f$ and $g$ into $3\times \delta^{-1-\epsilon'}$-balls as follows: set $f_l =f\chi_{B(l,3\delta^{-1-\epsilon'})}$ and $g_l =g\chi_{B(l,3\delta^{-1-\epsilon'})}$ for $l\in \delta^{-1}\BBZ^d.$ Then for $x\in B(l,\delta^{-1-\epsilon'})$
\[ |S_{\rho,\delta}^{\phi_1}(f-f_l)(x)|\lesssim \delta E\ast |f|(x)\quad \mbox{and}\quad |S_{\varrho-\rho,\delta}^{\phi_2}(g-g_l)(x)|\lesssim \delta E\ast |g|(x)\]
for all $\rho\in [0,1]$ where $E(x)=\delta^{\epsilon K}(1+\delta|x|)^{-d-1}$ for any $K>0$ and the implicit constant depends on $K$.   Also note from Lemma \ref{lem:kernel} that  the convolution kernels of  $S_{\rho,\delta}^{\phi_1}$, $S_{\varrho-\rho,\delta}^{\phi_2}$ are bounded by $\mathfrak K(x):=C\delta(1+\delta |x|)^{-N}$  for any $N$. 
Thus, writing $S_{\rho, \delta}^{\phi_1} fS_{\varrho-\rho, \delta}^{\phi_2} g= S_{\rho, \delta}^{\phi_1} f_l  S_{\varrho-\rho, \delta}^{\phi_2} g_l+   S_{\rho, \delta}^{\phi_1} (f- f_l)S_{\varrho-\rho, \delta}^{\phi_2} g+   S_{\rho, \delta}^{\phi_1} f_l  S_{\varrho-\rho, \delta}^{\phi_2} (g -g_l) $ and using the above we see that, if $x\in B(l,\delta^{-1-\epsilon'})$, 
\[| \CB_{\delta, \varrho}^{\phi_1, \phi_2}(f,g)(x)| \lesssim  |\CB_{\delta, \varrho}^{\phi_1, \phi_2}(f_l,g_l)(x)| + (E\ast |f|)(x) (\mathfrak K \ast |g|)(x)+(\mathfrak K \ast |f|)(x) (E \ast |g|)(x).\] 
Since we can take $K$ arbitrarily large, the contribution from the last two terms in the right hand side  is negligible. Thus, it is sufficient to show 
\[   \big(\sum_{l\in \delta^{-1}\BBZ^d} \|\CB_{\delta, \varrho}^{\phi_1, \phi_2}(f_l,g_l)\|_{L^r(B(l,\delta^{-1-\epsilon'}))}^r \big)^\frac1r\lesssim   \delta^B  \delta^{d(\frac1p+\frac1q-\frac1r- \frac1{p_0} -\frac1{q_0}+\frac1{r_0})-\epsilon}\|f\|_{p}\|g\|_{q}.
\] 
Using the assumption $  \|\CB_{\delta, \varrho}^{\phi_1, \phi_2}(f,g)\|_{r_0} \lesssim   \delta^B \|f\|_{p_0}\|g\|_{q_0}$
and H\"older's inequality give 
\[  \|\CB_{\delta, \varrho}^{\phi_1, \phi_2}(f_l,g_l)\|_{L^r(B(l,\delta^{-1-\epsilon'}))} \lesssim  \delta^{-C\epsilon'}  \delta^B  \delta^{d(\frac1p+\frac1q-\frac1r- \frac1{p_0} -\frac1{q_0}+\frac1{r_0})}\|f_l\|_{p}\|g_l\|_{q}.\] 
Since $1/p+1/q\ge 1/r$,  by H\"older inequality again for summation along $l$, 
\begin{align*}
 \big(\sum_{l\in \delta^{-1}\BBZ^d} &\|\CB_{\delta, \varrho}^{\phi_1, \phi_2}(f_l,g_l)\|_{L^r(B(l,\delta^{-1-\epsilon'}))}^r \big)^\frac1r\\
 & \lesssim    \delta^{-C\epsilon'+B+d(\frac1p+\frac1q-\frac1r- \frac1{p_0} -\frac1{q_0}+\frac1{r_0})} \big(\sum_{l\in \delta^{-1}\BBZ^d} \|f_l\|_{p}^p \big)^\frac1p \big(\sum_{l\in \delta^{-1}\BBZ^d} \|g_l\|_{q}^q\big)^\frac1q. 
\end{align*}
This gives the desired bound if we take $\epsilon'=\epsilon/C$ with large enough $C$. 
\end{proof}

The following is a bilinear version of  Proposition \ref{prop:s-t}.

\begin{lemma} \label{bibi}  Let $0<\delta\ll 1$, $\varrho\in [1/2,2]$,{ and $r\ge \frac{d-1}{d+1}$}. Then, for $\epsilon>0$ there is $N=N(\epsilon)$ such that 
\Be\label{bi-l2}
 \|\CB_{\delta, \varrho}^{\phi_1, \phi_2}(f,g)\|_{r} \lesssim 
\begin{cases} 
 \delta^{1-\epsilon} \|f\|_2\|g\|_2  &      \text{ if }  \frac1r\le \frac1{r_1}+\frac 1{r_2},  \\
\delta^{\frac{d^2+d+1}{2(d+1)}-\frac d{2r}} \|f\|_2  \|g\|_2  &  \text{ if }  \frac1{r_1}+\frac 1{r_2} \le \frac1r\le \frac2{r_1}\\
\end{cases}
\Ee 
holds uniformly {in $\delta,\varrho$, and $\phi_1,\phi_2$,} whenever $\phi_1, \phi_2\in \CC_N([-1,1])$. 
\end{lemma}

\begin{proof}  We start with observing the following:   For $2\le s\le r\le   \frac{2(d+1)}{d-1}$
and for any $\epsilon>0$ and $0<\delta\ll 1$, 
\begin{equation}
\label{eq:tem1}\Big\|\Big(\sum_{\rho\in\delta\Z\cap[0,1]} |S_{\rho,\delta}^\phi f(x)|^2\Big)^{1/2}\Big\|_r\lesssim 
\left\{ \begin{array}{ll} \delta^{-\beta_*(\frac1s)-\epsilon}\|f\|_s, &\mbox{if }~ \frac{d-1}d>\frac1s+\frac1r, \\
\delta^{-\frac d2(\frac1r-\frac1s)-\epsilon}\|f\|_s,  &\mbox{if }~ \frac{d-1}d\le\frac1s+\frac1r .
\end{array}\right.
\end{equation}
Indeed, by \eqref{scaled} and  Proposition \ref{prop:s-t}  we  have  that, for $2\le s\le r\le   \frac{2(d+1)}{d-1}$   and $k\ge 0$ and $\epsilon>0$,  
\[\Big\|\Big(\sum_{\rho\in \delta\Z\cap \BI_k}|S_{\rho,\delta}^\phi f|^2\Big)^{1/2}\Big\|_r\le C_\epsilon \delta^{-\beta_*(\frac1s)-\epsilon}2^{-k(\beta_*(\frac1s)+\frac d2(\frac1s-\frac 1r)+\epsilon)}\|f\|_s. \] 
Note that $\beta_*(\frac1s)+\frac d2(\frac1s-\frac 1r)=\frac{d-1}2-\frac{d}2(\frac1s+\frac1r)<0$ if $\frac{d-1}d<\frac1s+\frac1r$,  and $  \beta_*(\frac1s)+\frac d2(\frac1s-\frac 1r)\ge 0$ if $\frac{d-1}d\ge\frac1s+\frac1r$. Taking sum over $k$, we have \eqref{eq:tem1}.  Particularly, with $s=2$ we have 
\Be \label{l22}\Big\|\Big(\sum_{\rho\in\delta\Z\cap[0,1]} |S_{\rho,\delta}^\phi f(x)|^2\Big)^{1/2}\Big\|_r\lesssim 
\left\{ \begin{array}{ll} \delta^{\frac12-\epsilon} \|f\|_2, &\mbox{if }~ \frac{d-2}{2d}>\frac1r, \\
\delta^{\frac{d}{4}-\frac {d}{2r}-\epsilon}\|f\|_2,  &\mbox{if }~ \frac{d-2}{2d}\le \frac1r  \le \frac{d-1}{2(d+1)}\,.
\end{array}\right.
\Ee

Let us set $I_1=[0,2^{-4}]\cap [0,1]$, $I_2=[2^{-4}, \varrho-2^{-4} ]\cap[0,1] $, $I_3=[ \varrho-2^{-4}, \varrho+2^{-4} ]\cap[0,1]$,   $I_4=[  \varrho+2^{-4}, \infty)\cap[0,1]$. Depending on $\varrho$, $I_3$ and $I_4$ can be  an empty set. For $i=1,\dots, 4$, we set 
\[   \CB_i(f,g)=  \sum_{\rho\in\delta\mathbb Z\cap I_i}   |S_{\rho,\delta}^{\phi_1} f S_{\varrho-\rho,\delta}^{\phi_2}g|.\] 
Thus, we have
\begin{align*}
  |\CB_{\delta, \varrho}^{\phi_1, \phi_2}(f,g)|\le     \sum_{i=1}^4 \CB_i  (f,g). 
    \end{align*}
Note that if $\rho\in I_4$ then $\varrho -\rho \le -2^{-4}$, hence the Fourier support of $S^{\phi_2}_{\varrho-\rho,\delta}g$ is an empty set and $\CB_4(f,g)\equiv 0$ if $0<\delta<2^{-4}.$ Thus it is enough to deal with $\CB_1,\CB_2$, and $\CB_3$.
$\CB_2$ can be handled by using the estimates in Proposition \ref{prop:s-t}. In fact, 
 \[ 
 \CB_2(f,g)\le  \mathfrak D_a(f) \mathfrak D_b(g) 
  \] 
  where
\[   \mathfrak D_a (f) := \Big(\sum_{\rho\in \delta\Z\cap [2^{-4}, 1]}|S_{\rho,\delta}^{\phi_1} f|^2\Big)^{\frac12}, \   
  \mathfrak D_b(g) :=\Big(\sum_{\rho\in \delta\Z\cap [2^{-4},\varrho-2^{-4}]}|S_{\varrho- \rho,\delta}^{\phi_2 }g|^2\Big)^{\frac12}.    \] 
Since all the radii appearing in  $ \mathfrak D_a (f)$ and $ \mathfrak D_b (g)$ are $\sim 1$,  from a slight modification of proof of Proposition \ref{prop:s-t},  it is easy to see that $\mathfrak D_a (f) $  and  $\mathfrak D_b(g)$ satisfy the same estimate for $ \mathfrak{D}^\phi_\delta f$ which is in Proposition \ref{prop:s-t}.  
Thus,  using $L^2\to L^r,$ $r\ge  \frac{2(d+1)}{d-1}$ estimates for $\mathfrak D_a (f) $  and  $\mathfrak D_b(f)$ and  H\"older's inequality we see that, for $r\ge \frac{d+1}{d-1}$ and $\epsilon>0$,  
\[   \|\CB_2(f,g) \|_{r} \lesssim   \delta \|f\|_2\|g\|_2. \]  
This estimate is acceptable in view of the desired estimate. Hence, we are reduced to handling $\CB_1, \CB_3$ which are of similar nature.

We only handle $\CB_3$ since $\CB_1$ can be handled similarly. Now we note that 
 \[ 
 \CB_3(f,g)\le  \mathfrak D_c(f) \mathfrak D_d(g), 
  \] 
  where
\[   \mathfrak D_c (f) := \Big(\sum_{\rho\in \delta\Z\cap [2^{-3}, 1]}|S_{\rho,\delta}^{\phi_1} f|^2\Big)^{\frac12}, \   
  \mathfrak D_d(g) :=\Big(\sum_{\rho\in \delta\Z\cap [ \varrho-2^{-4}, \varrho+2^{-4} ]}|S_{\varrho- \rho,\delta}^{\phi_2 }g|^2\Big)^{\frac12}.    \] 
$ \mathfrak D_c (f) $ enjoys the same estimates  for $\mathfrak D_a (f)$ and $\mathfrak D_b (f)$ since the associated radii are $\sim 1$, whereas there are small radii  in  $\mathfrak D_d(g)$. It is easy to see that the estimate \eqref{l22} also holds for  $\mathfrak D_d(g)$. If $1/r\le 1/r_1+1/r_2$, then we may choose $\widetilde r_1$, $\widetilde r_2$ such that  $1/\widetilde r_1+1/ \widetilde r_2=1/r  $, and $ \widetilde r_1\ge r_1$, $\widetilde r_2\ge r_2$.  So using  $L^2\to L^r,$  $r\ge  \frac{2(d+1)}{d-1}$ estimates for $\mathfrak D_c (f) $ and the first estimate in \eqref{l22} for $\mathfrak D_d (g) $, we get 
\[   \|  \CB_3(f,g)\|_{r} \lesssim    \|  \mathfrak D_c (f) \|_{\widetilde r_1} \| \mathfrak D_d (g)\|_{\widetilde r_2}\lesssim \delta^{1-\epsilon} \|f\|_2\|g\|_2.     \] 
If $1/r_1+1/r_2<1/r\le 2/r_1$, we  take  $\widetilde r_1=r_1$ and $\widetilde r_2$ such that $1/\widetilde r_2=1/r-1/r_1$. Thus 
$r_1\le \widetilde r_2 < r_2$. Similarly as before,  using  both cases in \eqref{l22} we obtain 
\begin{align*}
   & \qquad \|  \CB_3(f,g)\|_{r}  \lesssim    \|  \mathfrak D_c (f) \|_{\widetilde r_1} \| \mathfrak D_d (g)\|_{\widetilde r_2}  \\ 
  \lesssim & 
\delta^\frac12  \delta^{\frac d4-\frac d2(\frac1r-\frac1{r_1})} \|f\|_2  \|g\|_2=\delta^{\frac{d^2+d+1}{2(d+1)}-\frac d{2r}} \|f\|_2  \|g\|_2.  
\end{align*} 
By the same argument as before it is easy to see that the same estimates also hold for $\CB_1$.  This completes the proof. 
\end{proof}

Finally we prove Theorem \ref{thm:sub} by making use of Lemma \ref{lem:local} and Lemma \ref{bibi}.

\begin{proof} [Proof of Theorem \ref{thm:sub}]
It is easy to see that $\gamma(p,q,r)\ge-1$ for $2\le p,q\le \infty$ and $r\ge \frac{d+1}{d-1}.$ Thus, combining Lemma \ref{lem:local} 
and Lemma \ref{bibi}, we have that,  for $2\le p,  q\le  \infty $ and $r\ge \frac{d+1}{d-1}$ satisfying $1/p+1/q\ge 1/r$,
\[\|\CB_{\delta, \varrho}^{\phi_1, \phi_2}(f,g)\|_{r} \lesssim   \delta^{-\gamma(p,q,r)-\epsilon}\|f\|_{p}\|g\|_{q}.\]
Since $\gamma(p,q,r)\ge-1$, we use  Lemma \ref{Lemm:reduction}  and \eqref{del-decomp} to obtain all the estimates in  Theorem \ref{thm:sub}. 
\end{proof}

\subsection{Lower bound for smoothing order $\alpha$ }\label{necessary}   {Similarly, as in case of  linear multiplier operator,} bilinear multiplier operator also have kernel expressions. We write $\CB^\alpha$ as 
\begin{equation}
\label{express}
 \CB^\alpha(f,g)(x)=\iint K^\alpha(x-y,x-z)f(y)g(z)dydz,\quad f,g\in \CS(\R^d),
 \end{equation}
where $K^\alpha=\CF^{-1}((1-|\xi|^2-|\eta|^2)^\alpha_+) $. Note that $K^\alpha$ is the kernel of the Bochner-Riesz operator $\CR^\alpha_1$ in $\R^{2d}$. From  the estimate for $K^\alpha$ in $\R^{2d}$ and duality, the necessary condition for $\CR^\alpha_1$ was obtained. Similar idea was used in \cite{b-g-s-y} to find some necessary conditions on $p,q$ for the boundedness of the operator $\mathcal B^\alpha$.
\begin{prop}\cite[Proposition 4.2]{b-g-s-y}\label{prop:ne} Let $1\le p,q\le \infty$ and $0<r\le \infty$ with $\frac1r=\frac1p+\frac1q$. \\
(i) If $\alpha\le d(\frac1r -1)-\frac12$, then $\CB^\alpha$ is unbounded from $L^p(\R^d)\times L^q(\R^d)$ to $L^r(\R^d).$\\
(ii) If $\alpha\le d |\frac1p-\frac12|-\frac12,$ then $\CB^\alpha$ is unbounded from $L^p(\R^d)\times L^\infty(\R^d)$ to $L^p(\R^d)$, from $L^\infty(\R^d)\times L^p(\R^d)$ to $L^p(\R^d)$, and also from $L^p(\R^d)\times L^{p'}(\R^d)$ to $L^1(\R^d)$ for each $1\le p\le \infty.$ 
\end{prop}
The first result in Proposition \ref{prop:ne} follows from the decay of kernel, since $\CB^\alpha(f,g)=K^\alpha$ if $\widehat{f}=\widehat{g}=1$ on $B(0,2)$. The second one  is a simple consequence of the linear theory.  Unfortunately, these results do not give meaningful necessary condition for the Banach case, since $d(\frac1r -1)-\frac12<0$ when $r\ge 1.$ 
The following  gives better lower bound.

\begin{prop}\label{prop:nc}
 Let $1\le p,q,r\le \infty$. If $\CB^\alpha$ is bounded from  $L^p(\R^d)\times L^q(\R^d)$
to $ L^{r}(\R^d)$, then 
\begin{align*}
\alpha 
&\ge\max\Big\{ \frac{d-1}2-\frac d{p}-\frac d{2q},\,\frac{d-1}2-\frac dq-\frac d{2p},\, 0\Big\}.
\end{align*}
\end{prop}
\begin{proof}
Let $\psi_\epsilon(\xi,\eta)=\phi_1(\xi/\epsilon)\phi_2(
\eta'/\epsilon)\phi_3((1-\eta_d)/\epsilon)$ where
$\eta'=(\eta_1,\dots, \eta_{d-1})$ and $\phi_1, \phi_2, \phi_3$
are nontrivial smooth functions supported in $B(0,1)$. Then $L^p\times L^q \to
L^{r}$ boundedness of $\CB^\alpha$ implies $L^p\times L^q \to
L^{r}$ boundedness of $\widetilde \CB^\alpha$ defined by
\[\widetilde \CB^\alpha(f,g)=
\int e^{2\pi ix\cdot(\xi+\eta)} \psi_\epsilon(\xi,\eta)(1-|\xi|^2-|\eta|^2)_+^\alpha\, \widehat f(\xi)\widehat g(\eta)\, d\xi d\eta.\]
 We
first note that
\begin{align*}
&\qquad\qquad\qquad\qquad \int \widetilde \CB^\alpha(f,g)(x) \phi(x) dx=\\
&\iint \iint \psi_\epsilon(\xi,\eta) (1-|\xi|^2-|\eta|^2)_+^\alpha\, \phi^\vee (\xi+\eta)
e^{-2\pi i(y\cdot\xi+z\cdot\eta)} d\xi d\eta f(y)g(z)dy dz.\end{align*} Choosing a
Schwartz function $\phi$  such that $\widehat \phi=1$ on $B(0,\sqrt
2)$, it follows that
\[\int \widetilde \CB^\alpha(f,g)(x) \phi(x) dx= \iint  \CK^\alpha(y,z) f(y)g(z)dy dz,\]
where
$\CK^\alpha=\CF^{-1}\big(\psi_\epsilon(\xi,\eta)(1-|\xi|^2-|\eta|^2)_+^\alpha\big)$.
Hence, $L^p\times L^q\to L^r$ boundedness of
$\widetilde\CB^\alpha$ implies
\begin{equation}\label{eq:1}
\Big|\int \CK^\alpha(y,z) f(y)g(z) dydz\Big|\lesssim \|f\|_p\|g\|_q.
\end{equation}
We choose a small enough $\epsilon>0$. By making use of  stationary
phase method (in fact, Fourier transform of measure supported in
sphere, for example, see \cite[p.68]{sogge}), for $w=(y,z)$ in a narrow conic neighborhood
$\mathcal C$ of $(0,e_d)\in \mathbb R^d\times \mathbb R^d$, to say
$\mathcal C=\{(y,z): \sqrt{|y|^2+|z'|^2}\le \epsilon_0 z_d\}$
for a small enough $\epsilon_0$,
\[\CK^\alpha(w)= {e^{i|w|} a(w)}{|w|^{-\frac{2d+1}{2}-\alpha}},\]
where $a$ is radial and  $|a(w)|\ge c>0$ if $|w|$ is large enough.
Let $R\gg \epsilon_0^{-100}$ and
set $A_R=\{x: (\epsilon_0/10)R^{1/2}\le
|x|< (\epsilon_0/5)R^{1/2}\}$ and $B_R=\{x: (\epsilon_0/10)R\le |x|\le (\epsilon_0/5)R,\, |x'|\le (\epsilon_0/10)|x_d| \}.$ Then  $A_R\times B_R\subset
\mathcal C$. We now set
\[f(y)=\chi_{A_R}(y) \,, \quad g(z)= \chi_{B_R} (z)e^{-i|z|}.\]
Thus,
\begin{equation*}
\Big|\int \CK^\alpha(y,z) f(y)g(z) dydz\Big|
=\Big|\int_{A_R}\int_{B_R}e^{i(|w|-|z|)}
a(w){|w|^{-\frac{2d+1}{2}-\alpha}}  dydz\Big|\gtrsim
R^{\frac{d-1}2-\alpha}
\end{equation*}
because $||w|-|z||=O(|y|^2/|z|)\le 1/4$ for large $R$. 
Moreover, by \eqref{eq:1}  we get
\[ \Big|\int \CK^\alpha(y,z) f(y) g(z) dydz\Big|\lesssim R^\frac d{2p} R^\frac dq.\]
 Combining the above two estimates and \eqref{eq:1}, the inequality $R^{-\alpha+\frac{d-1}2}\lesssim R^\frac d{2p} R^\frac dq$ should hold for any $R\gg \epsilon_0^{-100}$. Letting $R\to \infty$ gives $\alpha\ge\frac{d-1}2-\frac d{2p}-\frac dq$. If we exchange the role of $\xi$ and $\eta$ in the function $\psi_{\epsilon}$ (i.e., $\psi_{\epsilon}(\eta,\xi)$ instead of $\psi_{\epsilon}(\xi,\eta)$), we have the other condition $\alpha\ge\frac{d-1}2-\frac d{p}-\frac d{2q}$.
\end{proof}

\end{section}


\bibliographystyle{amsalpha}

\end{document}